\documentclass[reqno]{amsart}
\usepackage{amsmath, amsthm, amscd, amssymb, amsfonts, amsbsy}
\usepackage{latexsym, color, enumerate}
\usepackage{mathrsfs}
\usepackage{pxfonts}
\usepackage{bbm}
\usepackage[dvipsnames]{xcolor}

\usepackage{enumitem}

\usepackage{todonotes}
\usepackage{marginnote}

\theoremstyle{plain}
\newtheorem{theorem}[equation]{Theorem}
\newtheorem{lemma}[equation]{Lemma}

\newtheorem{proposition}[equation]{Proposition}

\theoremstyle{definition}

\theoremstyle{remark}
\newtheorem{remark}[equation]{Remark}

\newcommand{\dv}{\operatorname{div}}

\newcommand{\dist}{\operatorname{dist}}
\newcommand{\diam}{\operatorname{diam}}

\newcommand{\tr}{\operatorname{tr}}

\numberwithin{equation}{section}

\newcommand{\bR}{\mathbb{R}}

\providecommand{\set}[1]{\{#1\}}

\providecommand{\abs}[1]{\lvert#1\rvert}
\providecommand{\Abs}[1]{\left\lvert#1\right\rvert}
\providecommand{\bigabs}[1]{\bigl\lvert#1\bigr\rvert}

\providecommand{\norm}[1]{\lVert#1\rVert}

\renewcommand{\vec}[1]{\boldsymbol{#1}}

\def\dVert{\,\,\text{--}\kern-.46em\|}

\begin{document}
\title[Hopf-Oleinik Lemma of double divergence form equation]
{Hopf-Oleinik lemma for elliptic equations in double divergence form}

\author[H. Dong]{Hongjie Dong}
\address[H. Dong]{Division of Applied Mathematics, Brown University,
182 George Street, Providence, RI 02912, United States of America}
\email{Hongjie\_Dong@brown.edu}
\thanks{H. Dong was partially supported by the NSF under agreement DMS-2055244.}

\author[S. Kim]{Seick Kim}
\address[S. Kim]{Department of Mathematics, Yonsei University, 50 Yonsei-ro, Seodaemun-gu, Seoul 03722, Republic of Korea}
\email{kimseick@yonsei.ac.kr}
\thanks{S. Kim is supported by the National Research Foundation of Korea (NRF) under agreement NRF-2022R1A2C1003322.}

\author[B. Sirakov]{Boyan Sirakov}
\address[B. Sirakov]{Departamento de Matem\'atica, Pontificia Universidade Cat\'olica do Rio de Janeiro, Rio de Janeiro, Brazil}
\email{bsirakov@puc-rio.br}



\begin{abstract}
We establish, for the first time, a Zaremba-Hopf-Oleinik type boundary point lemma for uniformly elliptic partial differential equations in double divergence form, also known as stationary Fokker-Planck-Kolmogorov equations. As an application, we derive sharp two-sided estimates for the Green's function associated with second-order elliptic equations in non-divergence form in $C^{1,\alpha}$ domains.
\end{abstract}
\maketitle

\section{Introduction and main results}

The classical boundary point principle (BPP) states that a non-constant function $u$ which is superharmonic in a sufficiently smooth domain $\Omega$ (say a ball) and attains a minimum at $x_0\in \partial\Omega$ has a non-zero gradient (normal derivative) at $x_0$, in the sense that (denoting with $\vec \nu$ the interior normal to $\partial\Omega$)
\begin{equation}\label{hop1}
\liminf_{t\searrow0} \frac{u(x_0+t\vec \nu)-u(x_0)}{t}>0.
\end{equation}
An important alternative way of stating the BBP is to say that for superharmonic functions positivity entails a quantitative version of itself; specifically, a positive superharmonic function in a bounded $C^{1,1}$-domain $\Omega\subset \mathbb{R}^d$ is such that 
\begin{equation}\label{hop2}
u\ge c_0 \mathtt{d}\;\text{ in }\;\Omega,\; \text{ where }\;\mathtt{d}(x):=\mathrm{dist}(x,\partial\Omega),\;\;c_0= c_0(u,\Omega)>0.
\end{equation}
This ``bedrock'' (to quote the first page of the book \cite{PS}) result in the theory of elliptic PDE is most often associated with the names of Hopf \cite{Ho} and Oleinik \cite{Ol}.
Its study stretches from 1910 (\cite{Za}) to this day and remains a very active field, in particular its validity,  extensions and applications to various types of differential operators and domain geometries. In this paper we establish the framework for the validity of the Hopf-Oleinik lemma for \emph{double divergence form} equations.

The theory of linear second-order uniformly elliptic PDE has mostly considered the following types of operators 
\begin{align*}
\rm{(ND)}\qquad &L_0u=a^{ij} D_{ij}u = \tr(\mathbf{A} D^2 u)& \mbox{(non-divergence form)}%
\\
\rm{(D)}\qquad &\mathscr{L}_0 u=   D_i(a^{ij}D_ju) = \dv(\mathbf{A} D u)& \mbox{(divergence form)}
\\
\rm{(DD)}\qquad & L_0^*u=  D_{ij}(a^{ij} u) = \dv^2 (\mathbf{A} u)& \mbox{(double divergence form)}
\end{align*}
where $\mathbf{A}(x) = (a^{ij}(x))_{i,j=1}^d$ is such that $0<\lambda \mathbf{I} \le \mathbf{A} \le \Lambda \mathbf{I}$ on $\Omega\subset\mathbb{R}^d$, as well as more general operators with lower-order terms, as below.
Note that for symmetric $\mathbf{A}$ the operator (D) is {\it formally self-adjoint}, while (DD) is the {\it formal adjoint of} (ND).

There is a huge body of work on non-divergence and divergence form equations, which constitutes the very fundament of modern theory of elliptic and parabolic PDE. The literature on double divergence (DD) form equations, while quite large per se, is relatively small compared to that concerning (ND) and (D).
Apart from being instrumental in the study of (ND) as adjoints, (DD) include the important stationary Kolmogorov equation for invariant measures of diffusion processes (\cite{BKRS15}), and play a basic role in mean-field games theory since one of the equations in a mean-field system is of Fokker-Planck-Kolmogorov (FPK) type (\cite{GPV16}, \cite{CP20}).
Further applications of (DD) include homogenization (\cite{BBL18}) and more generally theory of nonlinear FPK equations (\cite{BRS19}).
We refer to the monograph \cite{BKRS15} for an extensive review and a lot of references on the theory of double divergence (FPK) equations. 

Classical works on (DD) include their basic axiomatic treatment in \cite{He}, \cite{Sjogren75}, the validity of the strong maximum principle in \cite{Li}, and the H\"older regularity of distributional solutions of equations with H\"older continuous coefficients $\mathbf{A}(x)$, proved in \cite{Sjogren73}.
The importance of studying (DD) in order to understand properties of its adjoint (ND) was further recognized in the influential works \cite{Ba1} and \cite{FS84}.

The regularity theory of double divergence form equations has reached certain maturity only recently.
Let us immediately note the principal difference between (DD) and the two other types of equations -- there is no gain of regularity for solutions of (DD) compared with the regularity of $\mathbf{A}$; a standard example is the one-dimensional equation $(a(x)u)^{\prime\prime}=0$, the solution $u=1/a$ is no more regular than $a$ itself.
Another fundamental difference between (DD) and (ND)-(D) is that non-zero constants are not solutions of $L^*u=0$, that is, adding a constant to a solution of (DD) does not yield a solution.
This precludes transferring to (DD) many known techniques from the study of (ND)-(D).
Also, standard a priori estimates such as the generalized maximum principle (ABP inequality, \cite[Theorem 9.1]{GT}) are not available for (DD). 

Regularity theory for (DD) studies for instance the question whether and under what conditions on the coefficients measure solutions of (DD) are locally integrable functions, and whether these functions have higher integrability, or are bounded and even have (H\"older) continuous representatives.
For the existence and integrability theory of solutions of double divergence form equations we refer to \cite{BKRS15}, \cite{BS17}, \cite{BRS23}, as well as the references in these works. 

Since the boundary point principle concerns continuous functions, we next review results on continuity and differentiability of solutions of double divergence form equations.
First, if $\mathbf{A}$ is Lipschitz, the equation can be differentiated into a divergence form one, so the well-known divergence theory applies.
Thus (DD) stands on its own only for less regular matrices $\mathbf{A}$.
It was proved in \cite{Sjogren73}, \cite{Ba2} that if $\mathbf{A}$ is H\"older continuous then solutions of $L_0^*u=0$ are H\"older continuous too, while mere (uniform) continuity of $\mathbf{A}$ may not suffice to guarantee that all solutions are continuous functions.
An extension to equations with only Dini continuous $\mathbf{A}$ and sufficiently integrable lower-order coefficients was obtained in \cite{Sjogren73}, \cite{BS17}.
The most general to date result was obtained in \cite{DK17}, \cite{DEK18}, where it is shown that solutions of 
\begin{equation}		\label{mostgeneq}
L^*u=\dv^2 (\mathbf{A} u)-\dv(\vec bu)+cu = \dv^2 \mathbf f+\dv \vec g+h
\end{equation}
are continuous provided $\mathbf{A}$, $\mathbf{f}$ have Dini mean oscillation (this is a weaker condition than Dini continuity), $\vec b, \vec g\in L^{p_0}(\Omega)$, $c, h\in L^{p_0/2}(\Omega)$ for some $p_0>d$ (it is well known that this Lebesgue integrability assumption cannot be improved even if the second-order operator is the Laplacian).
The result in \cite{DEK18} is also up-to-the-boundary if $\Omega$ is $C^{1,1}$ smooth, that is, the solution of the appropriate boundary-value problem $u\in C(\overline{\Omega})$, with modulus of continuity controlled in particular by the Dini mean oscillation of $\mathbf{A}$.
Further, it was recently observed in \cite{LPS20} that the zero set of a solution to $L_0^*u=0$ has stronger regularity than expected, namely, for a H\"older continuous $\mathbf{A}$ a solution is H\"older up to any exponent less than one around a point where it vanishes.
This result was significantly strengthened in \cite{CDKK24}, where it was established that solutions are actually differentiable at such points, with a corresponding boundary result also holding.
Furthermore, \cite{CDKK24} considers the full equation \eqref{mostgeneq} in a $C^{1,\alpha}$-domain, under the aforementioned general assumptions on the coefficients.
Additionally, we note that Harnack inequality for nonnegative solutions of $L^*u=0$ has been proven in \cite{DEK18} and \cite{BRS23} (see also \cite{GK24}). 
\smallskip

Next we turn to the boundary point principle.
A very complete survey on the long and surprisingly rich history of this result for non-divergence and divergence form equations can be found in \cite{AN3} (we refer also to \cite{AM},  \cite{AN2}, \cite{AN3}, \cite{BM}, \cite{FG},  \cite{K1}, \cite{LXZ}, \cite{N1}, \cite{S1}, \cite{S2} for conditions which ensure the validity of the BPP, as well as more historical reviews and a large body of references).
In particular, the BPP is always valid for supersolutions of non-divergence form equations such as $L_0 u=0$ but this is not so for its divergence counterpart $\mathscr{L}_0 u=0$ -- for the latter BPP may fail even if $\mathbf{A}$ is continuous, however it holds if $\mathbf{A}$ has Dini mean oscillation (\cite{RSS23}).
For recent quantitative results see \cite{Sir1}, \cite{Sir3}, \cite{DJV24}, \cite{TL24}, \cite{DJ25}. Furthermore, we observe that in the last ten years the BPP has been intensively studied  in the framework of nonlocal equations modeled on the fractional Laplacian - and in that setting the BPP has confirmed its capacity to generate various and unexpected results (see \cite{DSV24}, as well as \cite{Ro16}, \cite{SV19}, \cite{CDP23}, \cite{OS25} and the references there).

As for double divergence form equations, differently from the other types, it is obvious that the BPP as stated in \eqref{hop1} fails.
Indeed, in the same one-dimensional example as above $(a(x)u)''=0$,  $u=1/a$, one can take a smooth $a$, $a(x)\in[\lambda,\Lambda]$, which attains a maximum at an endpoint of the domain and whose derivative vanishes at that point.
On the other hand, this counterexample cannot be used to disprove \eqref{hop2}; and for a double divergence form equation it is not obvious that \eqref{hop1} and \eqref{hop2} are equivalent, because the set of solutions is not invariant with respect to addition of a constant.
What we are going to prove in this paper is that \eqref{hop2} does hold for (DD), that is, for double divergence form equations positivity still implies a quantitative version of itself. 
\smallskip

We now give the precise assumptions and state our main results. 
We consider the elliptic operator $L^*$ of the form
\[
L^* u= D_{ij}(a^{ij}u) - D_i(b^i u)+cu=\dv^2 (\mathbf A u)-\dv(\vec bu)+cu.
\]
The operator $L^*$ is the formal adjoint of the elliptic operator in non-divergence form $L$ given by
\[
Lv= a^{ij} D_{ij}v +  b^i D_i  v +cv=\tr(\mathbf{A} D^2 v)+ \vec b \cdot Dv +cv.
\]

We say that $u$ is a continuous $L^*$-supersolution in $\Omega$ if it is continuous in $\Omega$ and satisfies the following condition:
For any subdomain $\Omega_0 \subset \Omega$ with a smooth boundary and any function $v \in W^{2,1}(\Omega_0) \cap W^{1,1}_0(\Omega_0)$ with $v \ge 0$, we have
\begin{equation}			\label{eq0710thu}
\int_{\Omega_0} u Lv  \le \int_{\partial \Omega_0} (Dv\cdot \vec n) u a^{ij}n_i n_j.
\end{equation}
To justify this definition, we formally integrate by parts in $\Omega_0$, using $v=0$ and  $Dv=(Dv\cdot \vec n)\vec n$ on $\partial \Omega_0$, to obtain
\[
\int_{\Omega_0} u Lv=\int_{\partial \Omega_0} (Dv\cdot \vec n) u a^{ij}n_i n_j +\int_{\Omega_0} v L^* u.
\]
Thus, condition \eqref{eq0710thu} formally implies that $L^*u \le 0$ in $\Omega_0$.

We assume that the coefficients $\mathbf{A}:=(a^{ij})$ are such that for some
constants $0<\lambda\le \Lambda$ 
\begin{equation}					\label{ellipticity-nd}
a^{ij}=a^{ji},\quad \lambda \abs{\xi}^2 \le a^{ij}(x) \xi^i \xi^j \le \Lambda \abs{\xi}^2,\quad x\in\Omega,\;\; \xi\in\mathbb{R}^d,\quad \mbox{ and}
\end{equation}
\begin{equation}					\label{DiniMean}
\mathbf{A}\mbox{ has Dini mean oscillation}.
\end{equation}
In addition we assume that the lower order coefficients $\vec b$ and $c$ are such that
\begin{equation}			\label{cond_lower}
\vec b \in L^{p_0}(\Omega) \;\text{ and }\;c \in L^{p_0/2}(\Omega)\;\text{ for some }p_0>d.
\end{equation}

We recall  that a function $f$ defined on $\Omega$ has Dini mean oscillation, and we write $f \in \mathrm{DMO}(\Omega)$ if the mean oscillation function $\omega_f: \mathbb{R}_+ \to \mathbb{R}$, defined by
\[
\omega_f(r):=\sup_{x\in \Omega} \fint_{\Omega \cap B_r(x)} \,\abs{f(y)-(f)_{\Omega \cap B_r(x)}}\,dy, \;\; \text{where }(f)_{\Omega \cap B_r(x)}=\fint_{\Omega \cap B_r(x)} f,
\]
is a Dini function. That is, it satisfies the Dini condition
\[
\int_0^1 \frac{\omega_f(t)}t \,dt <+\infty.
\]
It is clear that a Dini continuous function has Dini mean oscillation.
However the contrary is not true, the DMO condition is weaker than Dini continuity - see \cite{DK17}, \cite{MMcO2022} for examples.
Note that the DMO condition implies uniform continuity with modulus governed by $\omega_f$, see the appendix of \cite{HK20}. 
We also observe that our assumptions ensure that the function
\begin{equation}				\label{omega_coef}
\omega_{\rm coef}(r):=  \omega_{\mathbf A}(r)+  r\sup_{x \in \Omega} \fint_{\Omega \cap B_r(x)}\abs{\vec b} +r^2\sup_{x \in \Omega} \fint_{\Omega \cap B_r(x)}\abs{c}
\end{equation}
satisfies the Dini condition.

\begin{theorem}	\label{maintheorem}
Assume that $\Omega$ is a bounded $C^{1,\alpha}$ domain for some $\alpha>0$ and that $L^*$ satisfies \eqref{ellipticity-nd} -- \eqref{cond_lower} in $\Omega$.
Suppose $u \ge 0$ is a continuous $L^*$-supersolution in $\Omega$ that is not identically zero.
Then $u>0$ in $\Omega$.
Moreover, if there exists a point $x_0\in \partial \Omega$ such that $u(x_0)=0$, then
\[
\liminf_{t\searrow0}\, \frac{u(x_0+t\vec \nu)}{t}>0,
\]
where $\vec \nu$ denotes the interior unit normal to $\partial\Omega$ at $x_0$.
In particular, property \eqref{hop2} holds. 
\end{theorem} 

We stress that Theorem \ref{maintheorem} is completely new for double divergence form equations, even in the simplest form $L_0^* u=0$ with $\mathbf{A} \in C^\alpha$, $\alpha \in (0,1)$.
Nevertheless, our assumptions in Theorem \ref{maintheorem} on the equation and the domain are close to optimal.
Indeed, it is known that even for $L_0=\Delta$ the Lebesgue integrability in \eqref{cond_lower} cannot be improved, and also that the BPP does not hold in $C^1$ (as opposed to $C^{1,\alpha}$) domains.
In Section \ref{counter_ex}, we will give an example which shows that the DMO condition on $\mathbf{A}$ cannot be replaced by mere uniform continuity in Theorem \ref{maintheorem}.

The proof of Theorem \ref{maintheorem} relies on a (subtle) iteration argument and makes important use of the machinery developed in \cite{CDKK24}.

In the end we give an application of the above result to pointwise bounds for the Green function of a uniformly elliptic operator in non-divergence (ND) form.
We denote $a\wedge b=\min\{a,b\}$.

\begin{theorem}	\label{GreenEstimate}
Assume that $d\ge3$, $\Omega$ is a bounded $C^{1,\alpha}$ domain for some $\alpha>0$, and that the coefficients of $L$ satisfy  \eqref{ellipticity-nd} -- \eqref{cond_lower} in $\Omega$.
Additionally, suppose that $c \le 0$ in $\Omega$.
The Green's function $G(x,y)$ of the operator $L$ satisfies the bounds
\begin{multline}				\label{eq1850mon}
\frac{1}{C\abs{x-y}^{d-2}}\left(1\wedge\frac{\mathtt{d}(x)}{\abs{x-y}}\right)
\left(1\wedge\frac{\mathtt{d}(y)}{\abs{x-y}}\right)\le G(x,y)\\
\le \frac{C}{\abs{x-y}^{d-2}}\left(1\wedge\frac{\mathtt{d}(x)}{\abs{x-y}}\right)
\left(1\wedge\frac{\mathtt{d}(y)}{\abs{x-y}}\right),\quad x \neq y \in \Omega, 
\end{multline}
where $C>0$ is a constant depending only on $d$, $\lambda$, $\Lambda$, $\omega_{\mathbf A}$, $p_0$, $\norm{\vec b}_{L^{p_0}}$, $\norm{c}_{L^{p_0/2}}$, $\diam \Omega$, and the $C^{1,\alpha}$ characteristic of $\Omega$.
\end{theorem}

For divergence form operators this estimate goes back to the well-known works \cite{GW82} and \cite{Zh}. The existence of the Green function for $L_0$ under the hypotheses of Theorem \ref{GreenEstimate} was recently proven in \cite{HK20}, together with the crude estimate $G(x,y)\le C\abs{x-y}^{d-2}$.
The full estimate in the above theorem was subsequently obtained via probabilistic methods in \cite{CW23} (that paper contains also a review of the history and challenges of bounds on Green functions), and under the supplementary assumptions that $\mathbf{A}$ is Dini continuous and $\Omega$ is $C^{1,1}$ smooth.
We relax these hypotheses here, and give a pure PDE proof of the bound on the Green function.

We note that the upper bound in Theorem \ref{GreenEstimate} was previously established in \cite{CDKK24}. The lower bound follows from the Harnack inequality in \cite{GK24} and Proposition \ref{prop01}, which plays a crucial role in the proof of Theorem \ref{maintheorem}.

\section{Proof of Theorem \ref{maintheorem}}		\label{sec2}
We begin by observing that we may assume $c \le 0$.
Indeed, decomposing $c$ as $c=c^{+}-c^{-}$, we define the modified operator
\[
L_{-}^*=\dv^2(\mathbf A u)-\dv(\vec b u)-c^-u.
\]
Since $L^*u=L_{-}^*u+c^+u$ and $u \ge 0$ in $\Omega$, it follows that $u$ is an $L^*_{-}$-supersolution in $\Omega$.
Therefore, for the remainder of the proof, we assume $c \le 0$ by replacing $c$ with $-c^-$.

The following lemma establishes a comparison principle for $L^*$-supersolution.

\begin{lemma}[Comparison Principle]
Let $\tilde \Omega$ be a bounded smooth domain.
Suppose $u$ is a continuous $L^*$-supersolution in $\tilde\Omega$ and satisfies
\[
\liminf_{x\to x_0} u(x) \ge 0,\qquad \forall x_0 \in \partial \tilde\Omega.
\]
Then $u \ge 0$ in $\tilde\Omega$.
\end{lemma}
\begin{proof}
Consider $g \in C^\infty_c(\tilde\Omega)$ with $g \le 0$ in $\tilde\Omega$.
Let $v \in W^{2, p_0/2}(\tilde \Omega) \cap W^{1,p_0/2}_0(\tilde\Omega)$ be the solution to the problem
\[
Lv=g \;\text{ in }\;\tilde\Omega,\quad v=0\;\text{ on }\;\partial\tilde\Omega.
\]
Since $Lv =g \le 0$ in $\tilde\Omega$, the maximum principle implies that $v \ge 0$ in $\tilde\Omega$.
Moreover, since $v\ge 0$ in $\tilde\Omega$ and $v=0$ on $\partial\tilde\Omega$, we conclude that $(Dv\cdot \vec n) \le 0$ on $\partial\tilde\Omega$.
Thus, it follows from \eqref{eq0710thu} that
\[
\int_{\tilde\Omega} u g=\int_{\tilde \Omega} u Lv
\le \int_{\partial \tilde \Omega} (Dv\cdot \vec n) u a^{ij}n_i n_j \le 0.
\]
Since $g \le 0$ is arbitrary, we conclude that $u \ge 0$ in $\tilde \Omega$.
\end{proof}

We will show that $u>0$ in $\Omega$.
Since $u$ is not identically zero, there exists a point  $x^o \in \Omega$ such that $u(x^o)>0$.
For any $x\in \Omega$, we can construct a sequence of balls $\{B_i\}_{i=0}^N$ satisfying
\[
\overline B_i \subset \Omega,\quad B_i \cap B_{i+1} \neq \emptyset,\quad  x^o \in \partial B_0,\quad x \in B_N.
\] 
Let $w \in C(\overline B_0)$ be the solution of
\begin{equation}			\label{eq0740mon}
L^* w = 0\;\text{ in }\;B_0,\quad w=u\;\text{ on }\partial B_0.
\end{equation}
The existence of such a function $w$ follows from \cite[Lemma 2.22]{DEK18} and a simple decomposition argument as follows.
Consider the problem
\[
D_{ij}(a^{ij} w_1) = 0\;\text{ in }\;B_0,\quad w_1=u\;\text{ on }\partial B_0.
\]
By \cite[Lemma 2.22]{DEK18}, there exists a unique solution $w_1 \in C(\overline B_0)$.
In particular, this implies that $b^i w_1 \in L^{p_0}(B_0)$ and $c w_1 \in L^{p_0/2}(B_0)$.
Next, let $w_2 \in L^{p_0/2}(B_0)$ be the solution of 
\[
L^*w_2=D_i(b^iw_1)-cw_1\;\text{ in }\;B_0,\quad w_2=0 \;\text{ on }\;\partial B_0.
\]
The existence and uniqueness of $w_2$ follow from $W^{2,p}$ theory for $L$ and the  transposition method (see \cite{EM2017}).
Moreover, by \cite{DEK18}, we have $w_2 \in C(\overline B_0)$.
Thus, defining $w=w_1+w_2$ we obtain a solution to \eqref{eq0740mon}.

By the comparison principle, we have
\[
0\le w \le u\;\text{ in }\;B_0.
\]
Applying the Harnack inequality, we obtain $w>0$ in $B_0$, which implies $u>0$ in $B_0$.
Repeating this process iteratively along the sequence of balls $\{B_i\}$, we establish that $u>0$ in $B_N$, and in particular, $u(x)>0$.
Thus, we conclude that $u>0$ in $\Omega$.

\medskip

Next, we introduce the following notation:
\[
B_r^+=B_r^{+}(0)=B_r(0) \cap \set{x_d >0}\quad\text{and}\quad  T_r=T_r(0)=B_r(0) \cap \set{x_d=0}.
\]
We fix a smooth (convex) domain $\mathcal{D}$ satisfying
\begin{equation}			\label{eq1142wed}
B_{1/2}^{+}(0) \subset \mathcal{D} \subset B_1^{+}(0),
\end{equation}
such that $\partial \mathcal{D}$ contains a flat portion $T_{1/2}(0)$.
For $y_0 \in \partial \mathbb{R}^d_{+}=\set{x_d=0}$, we define
\[
B_r^{+}(y_0)=B_r^{+}(0)+ y_0,\quad T_r(y_0)=T_r(0)+y_0,\quad\text{and}\quad \mathcal{D}_r(y_0)= r \mathcal{D}+ y_0.
\]

\smallskip
We apply the \emph{boundary flattening method} using the regularized distance function, followed by an affine transformation, as described in Section 4.2 of \cite{CDKK24}.

We note that this transformation ensures the transformed function $\tilde u$ satisfies
\[
\tilde L^* \tilde u  \le 0, \quad \tilde u > 0\;\text{ in }\;B_8^+,\quad \tilde u(0)=0,
\]
where $\tilde{L}^*$ denotes the transformed operator.

As explained in \cite{CDKK24}, this transformation preserves the property that $\tilde\omega_{\rm coef}$, the corresponding function of $\omega_{\rm coef}$, remains a Dini function.
Moreover, the transformed lower-order term coefficients remain in $L^p$ for some $p>1$.
Thus, by Remark 1.8 in \cite{CDKK24} and as we will reaffirm in the proof below, it suffices to work with $\tilde L^*$ and $\tilde u$ in $B_8^+$ instead of $L$ and $u$ in a neighborhood of $x_0$. 

Let $v \in C(\overline{\mathcal D}_8)$ be a solution of the problem
\[
\tilde L^* v = 0\;\text{ in }\;\mathcal{D}_8,\quad  v=g\;\text{ on }\partial \mathcal{D}_8,
\]
where $g$ is a continuous function satisfying
\[
0 \lneqq g \le \tilde u\;\text{ on }\partial \mathcal{D}_8,\quad g = 0\;\text{ in }\;T_4.
\]

By the comparison principle applied to $\tilde u-v$ in $\mathcal D_8$, we obtain $\tilde u \ge v$ in $\mathcal D_8$.
Clearly, $v$ is not identically zero.
Therefore, Proposition \ref{prop01} asserts that
\[
D_d v(0) \ge \frac{C}{r_0} \fint_{B_{r_0}^+} v>0.
\]
Since $\tilde u(0)=v(0)=0$ and $\tilde u \ge v$ in $B_4^+$, we obtain
\[
\liminf_{t\searrow0} \frac{\tilde u(t \vec e_d)}{t} \ge \liminf_{t \searrow0} \frac{v(t \vec e_d)-v(0)}{t} =D_d v(0)>0.
\]
The theorem follows from transforming $\tilde u$ back to $u$.

For notational convenience, in the following proposition and throughout the rest of this section, we replace $\tilde u$ and $\tilde L^*$ with $u$ and $L^*$, respectively.

\begin{proposition}			\label{prop01}
There exist constants $C>0$ and $r_0 \in (0,\frac12)$ depending on $d$, $\lambda$, $\Lambda$, and $\omega_{\rm coef}$ such that
\begin{equation}			\label{eq2017sun}
D_d u (0) \ge \frac{C}{r_0} \fint_{B_{r_0}^+} u
\end{equation}
whenever $u$ is a nonnegative solution of
\begin{equation}			\label{eq1737mon}
L^*u =0\;\text{in }\; B_4^+,\quad u=0\;\text{ on }\;\partial B_4^+ \cap \set{x_d=0}.
\end{equation}
Furthermore, there exist constants $\tilde r_0 \in (0, \frac12)$ and $\beta_1>0$ depending on $d$, $\lambda$, $\Lambda$, and $\omega_{\rm coef}$ such that
\begin{equation}			\label{eq_bdry_hnk}
\inf_{B_r^+}\,\frac{u(x)}{x_d} \ge \frac{\beta_1}{r} \fint_{B_r^+} u,\quad \forall r \in (0,\tilde r_0).
\end{equation}
\end{proposition}

\medskip
It remains to prove Proposition \ref{prop01}.
We divide the proof into several steps.
\subsection*{Step 1}
Let $u$ be a nonnegative solutions of \eqref{eq1737mon}.
We begin by stating some useful properties of $u$.
\begin{lemma}[Doubling property]			\label{lem_dp}
There exists a constant $C>0$ depending only on $d$, $\lambda$,  $\Lambda$, $p_0$, $\norm{\vec b}_{L^{p_0}(\Omega)}$, $\norm{c}_{L^{p_0/2}(\Omega)}$, and $C^{1,\alpha}$ characteristics of $\partial\Omega$, such that for any $r \in (0,1]$,
\[
\fint_{B_{2r}^+} u \le C \fint_{B_r^+}u.
\]
\end{lemma}
\begin{proof}
If the operator has no lower-order terms, the lemma follows directly from \cite[Lemma 2.0]{FS84}. To extend the result to our setting, we eliminate the lower-order terms using the method developed in \cite{GK24}, which is based on the Zvonkin transformation. We begin by applying the boundary flattening procedure introduced at the start of this section. The resulting drift term $\vec{b}$ belongs to a suitable Morrey space, allowing us to employ the technique used in \cite{GK24}. Although \cite[Theorem 4.3]{GK24} is stated for $\vec{b} \in L^p$ with $p > d$, the same method applies when $\vec{b}$ lies in the Morrey space we obtain. See the proof of \cite[Theorem 1.1]{GK24}, the parabolic analogue of \cite[Theorem 4.3]{GK24}, which is formulated for lower-order coefficients in Morrey spaces.
\end{proof}

From the boundary $C^0$ estimates for $u$ (see Theorem 1.8 and  Lemma 2.26 in \cite{DEK18}), we obtain
\begin{equation}			\label{eq1234thu}
\norm{u}_{L^\infty(B_r^+)} \le C \fint_{B_{2r}^+} u,
\end{equation}
where $C=C(d, \lambda, \Lambda, \omega_{\rm coef})$.
\begin{remark}
From the context of \cite{DEK18}, the constant $C$ in \eqref{eq1234thu} appears to depend on $\omega_{\mathbf A}$, $\norm{\vec b}_{L^{p_0}}$, and $\norm{c}_{L^{p_0/2}}$, in addition to $d$, $\lambda$, and $\Lambda$.
However, a careful analysis shows that these dependencies are fully accounted for by the assumption that $\omega_{\rm coef}$ from \eqref{omega_coef} is a Dini function.
\end{remark}

Combining \eqref{eq1234thu} with Lemma \ref{lem_dp}, we obtain
\[
\norm{u}_{L^\infty(B_r^+)} \le C \fint_{B_r^+} u,
\]
where $C=C(d, \lambda, \Lambda, \omega_{\rm coef})$.
Thus, there exists $\varepsilon=\varepsilon(d, \lambda, \Lambda, \omega_{\rm coef}) \in (0,\frac12)$ such that
\[
\fint_{B_r^+} \mathbbm{1}_{\{x_d\le \varepsilon r\}} \,u \le \norm{u}_{L^\infty(B_r^+)} \, \frac{\abs{\{x_d \le \varepsilon r\} \cap B_r^+}}{\abs{B_r^+}} \le \frac{1}{2} \fint_{B_r^+} u.
\]

Define $U_r:=\set{x_d > \varepsilon r} \cap B^+_r$.
Then, we have
\[
\frac12 \fint_{B_r^+} u \le  \fint_{B_r^+} u -\fint_{B_r^+} \mathbbm{1}_{\{x_d \le \varepsilon r\}} \,u = \fint_{B_r^+} \mathbbm{1}_{\{x_d>\varepsilon r\}} u  \le \fint_{U_r} u.
\]
Applying the Harnack inequality (see, for example, \cite[Theorem 4.3]{GK24}) to $u$ in $U_r$, we obtain
\[
\inf_{U_r} u \ge \frac{1}{C} \sup_{U_r} u \ge \frac{1}{C} \fint_{U_r} u \ge  \frac{1}{C} \fint_{B_r^+} u.
\]
where $C=C(d,\lambda, \Lambda, \omega_{\rm coef})$.
We emphasize that $C$ is independent of $r \in (0,1]$.

Therefore, we obtain the following estimate:
\begin{equation}				\label{eq1532wed}
\left(\fint_{B_r^+} u^{\frac12} \right)^2 \ge
\left(\frac{\abs{U_r}}{\abs{B_r^+}}\fint_{U_r} u^{\frac12} \right)^2 \ge  \nu_0 \fint_{B_r^+} u,
\end{equation}
where $\nu_0 =\nu_0(d, \lambda, \Lambda, \omega_{\rm coef})>0$.

\subsection*{Step 2}
Next, we decompose $u$ into a sum of ``small'' and ``regular'' parts.

\smallskip
Let $\bar{\mathbf A}=(\mathbf A)_{B_{2r}^+}$ denote the average of $\mathbf A$ over $B_{2r}^+=B_{2r}^+(0)$.
We then express $u$ as $u=v+w$, where $w$ is the $L^p$ weak solution (for some $p>1$) of the problem
\begin{equation}			\label{eq0708mon}
\left\{
\begin{aligned}
\dv^2( \bar{\mathbf A} w) &= -\dv^2 ((\mathbf{A}-\bar{\mathbf A})u)+\dv (\vec b u)-cu\;\text{ in }\; \mathcal{D}_{2r},\\
(\bar{\mathbf A} \nu \cdot \nu) w&= -(\mathbf{A}-\bar{\mathbf A})u \nu \cdot \nu \;\mbox{ on }\;\partial \mathcal{D}_{2r}.
\end{aligned}
\right.
\end{equation}

\begin{lemma}			\label{lem01}
Let $B=B_1(0)$, and let $\bar{\mathbf A}$ be a constant symmetric matrix satisfying the condition \eqref{ellipticity-nd}.
For $\mathbf f \in L^p(B)$, $\vec g \in L^p(B)$, and $h \in L^p(B)$ for some $p>1$, let $u \in L^p(B)$ be the weak solution to the problem
\[
\left\{
\begin{aligned}
\dv^2(\bar{\mathbf A} u)&= \dv^2 \mathbf f+ \dv \vec g + h\;\mbox{ in }\; B,\\
(\bar{\mathbf A} \nu \cdot \nu) u &= \mathbf f \nu\cdot \nu \;\mbox{ on } \; \partial B.
\end{aligned}
\right.
\]
Then, for any $t>0$, we have
\[
\abs{\{x \in B : \abs{u(x)} > t\}}  \le \frac{C}{t} \left(\int_{B} \abs{\mathbf f}+ \abs{\vec g}+ \abs{h} \right),
\]
where $C=C(d, \lambda, \Lambda)$.
The conclusion remains unchanged if we replace $B_1(0)$ with $\mathcal{D}$.
\end{lemma}

\begin{proof}
See \cite[Appendix]{CDKK24}.
\end{proof}

Recall the definition of  $\omega_{\rm coef}$ in \eqref{omega_coef}.
Using Lemma \ref{lem01} and the geometric properties of $\mathcal{D}$, we obtain, after rescaling,
\begin{equation}			\label{eq2017fri}
\left(\fint_{B_{r}^+} \abs{w}^{\frac12}\right)^{2} \le C \omega_{\rm coef}(2r)\, \norm{u}_{L^\infty(B_{2r}^+)},
\end{equation}
where $C=C(d,\lambda, \Lambda)$.
In fact,  $w \in L^p(B_r^+)$ for any $p \in (0,1)$; however, we fix $p=\frac{1}{2}$ for definiteness.
Note that $v=u-w$ satisfies
\begin{equation}			\label{eq0406mon}
\dv^2(\bar{\mathbf A} v)=0\;\text{ in }\;\mathcal{D}_{2r},\quad
(\bar{\mathbf A} \nu \cdot \nu) v = (\mathbf{A} \nu \cdot \nu)u\;\text{ on }\;\partial \mathcal{D}_{2r}.
\end{equation}
In particular, observe that
\[
v=(\mathbf{A} \nu \cdot \nu)u /(\bar{\mathbf A} \nu \cdot \nu) \ge 0 \text{ on }\;\partial \mathcal{D}_{2r}.
\]
Hence, by the maximum principle, we find that
\begin{equation}			\label{eq0410mon}
v \ge 0\;\text{ in }\;\mathcal{D}_{2r}.
\end{equation}
Moreover, it follows from \eqref{eq0406mon} that $v$ satisfies
\begin{equation}			\label{eq0450mon}
\dv(\bar{\mathbf A} Dv)=0\;\text { in }\; B_r^+,\quad
v=0 \;\text{ on }\;T_r.
\end{equation}

By the regularity estimates for solutions of elliptic equations with constant coefficients, we conclude that $v \in C^{\infty} (\overline{B_{r/2}^+})$.
In particular, the following estimate holds:
\[
[v]_{C^{1,1}(\overline{B_{r/2}^+})} \le \frac{C}{r^2} \left(\fint_{B_r^+} \abs{v}^{\frac12}\right)^{2},
\]
where $C=C(d,\lambda, \Lambda)$.
Applying Taylor's theorem, we obtain
\[
\sup_{x\in B_\rho^+}\, \bigabs{v(x)-v(0)-Dv(0)\cdot x}  \le C [v]_{C^{1,1}(B_\rho^+)}\, \rho^2,\quad \forall \rho \in (0, \tfrac12 r].
\]
Consequently, we derive the estimate
\[
\left(\fint_{B_{\rho}^+} \bigabs{v(x)-v(0)-Dv(0)\cdot x}^{\frac12}dx\right)^{2} \le C [v]_{C^{1,1}(B_{r/2}^+)}\, \rho^2 \le C \left(\frac{\rho}{r}\right)^2 \left(\fint_{B_r^+} \abs{v}^{\frac12}\right)^{2}.
\]
This estimate remains valid when $v$ is replaced by $v-\ell$, where $\ell(x)=\vec p \cdot x +q$ is any affine function.
Hence, for any $\kappa \in (0, \frac12)$ and any affine function $\ell$, we obtain
\begin{equation}			\label{eq1616fri}
\left(\fint_{B_{\kappa r}^+} \bigabs{v(x)-v(0)-Dv(0)\cdot x}^{\frac12}dx\right)^{2}  \le C_0 \kappa^2 \left(\fint_{B_r^+} \abs{v-\ell}^{\frac12}\right)^{2}.
\end{equation}
where $C_0=C_0(d,\lambda, \Lambda)>0$.

Furthermore, from \eqref{eq0410mon} and \eqref{eq0450mon}, it follows that $v$ satisfies the boundary Harnack inequality:
\begin{equation}			\label{eq1414tue}
\inf_{B_{r/2}^+} \,\frac{v(x)}{x_d} \ge \frac{\beta_0}{r} \left(\fint_{B_r^+} \abs{v}^{\frac12} \right)^2,
\end{equation}
where $\beta_0$ is a positive constant depending only on $d$, $\lambda$, and $\Lambda$.
See, for instance, \cite[Theorem 2.1]{Sir3}.

We recall that the following quasi-triangle inequality holds for $L^{1/2}$ quasi-norms:
\begin{equation}			\label{qt}
\norm{v+w}_{L^{1/2}} \le 2 \norm{v}_{L^{1/2}}+2\norm{w}_{L^{1/2}}.
\end{equation}
From \eqref{eq2017fri}, \eqref{eq1234thu}, and Lemma \ref{lem_dp}, we obtain
\begin{equation}			\label{eq2050sat}
\left(\fint_{B_r^+} \abs{w}^{\frac12}\right)^{2} \le C \omega_{\rm coef}(2r) \fint_{B_r^+} u.
\end{equation}
Let us fix $r_1 \in (0,\frac12)$ such that
\begin{equation}			\label{eq0818mon}
C \omega_{\rm coef}(2r_1) \le \frac{\nu_0}{4},
\end{equation}
where $\nu_0$ is the constant from \eqref{eq1532wed}.
Then  \eqref{qt}, \eqref{eq1532wed}, and \eqref{eq2050sat} imply
\begin{equation}			\label{eq1556wed}
\left(\fint_{B_r^+} v^{\frac12} \right)^2 \ge \frac12 \left(\fint_{B_r^+} u^{\frac12} \right)^2 -\left(\fint_{B_r^+} \abs{w}^{\frac12} \right)^2 \ge \frac{\nu_0}{4} \fint_{B_r^+} u\quad \forall r\in (0,r_1).
\end{equation}

\subsection*{Step 3}

We introduce the function
\begin{equation}			\label{eq1100sat}
\varphi(r):=\frac{1}{r} \,\inf_{\substack{\vec p \in \mathbb{R}^d\\q \in \mathbb{R}}} \left(\fint_{B_{r}^+} \bigabs{u(x)-\vec p \cdot x - q}^{\frac12}dx\right)^{2}.
\end{equation}
Applying H\"older's inequality and using the assumption that $u \ge 0$, we obtain the bound
\begin{equation}			\label{eq1523sat}
\varphi(r) \le \frac{1}{r} \left(\fint_{B_r^+} \abs{u}^{\frac12}\right)^2 \le  \frac{1}{r} \fint_{B_{r}^+} u.
\end{equation}

We aim to control the function $\varphi^+(r)$ using an iterative method, as employed in the proof of \cite[Theorem 1.7]{CDKK24}.
To make this article more self-contained, we replicate most of the argument from \cite{CDKK24}.

For $r \in (0,1)$, we decompose $u$ as $u=v+w$, where $w$ solves the problem \eqref{eq0708mon} in the previous step.
Since $u=v+w$, we apply \eqref{qt}, \eqref{eq2017fri}, and \eqref{eq1616fri} to obtain
\begin{align}			\nonumber
\kappa r \varphi(\kappa r) & \le \left(\fint_{B_{\kappa r}^+} \bigabs{u(x)-Dv(0)\cdot x -v(0)}^{\frac12}dx\right)^{2}\\					\nonumber
&\le 2\left(\fint_{B_{\kappa r}^+} \bigabs{v(x)-Dv(0)\cdot x -v(0)}^{\frac12}dx\right)^{2} + 2 \left(\fint_{B_{\kappa r}^+} \abs{w}^{\frac12}\right)^{2}\\
					\nonumber
& \le  C_0 \kappa^2 \left(\fint_{B_r^+} \abs{v-\ell}^{\frac12}\right)^{2} + 2 \kappa^{-2d}\left(\fint_{B_r^+} \abs{w}^{\frac12}\right)^{2}\\	
					\nonumber
& \le 2C_0 \kappa^2 \left(\fint_{B_r^+} \abs{u-\ell}^{\frac12}\right)^{2} + \left(2C_0\kappa^2+2\kappa^{-2d}\right)\left(\fint_{B_r^+} \abs{w}^{\frac12}\right)^{2}\\
					\label{eq0224sat}
& \le 2C_0 \kappa^2 \left(\fint_{B_r^+} \abs{u-\ell}^{\frac12}\right)^{2} + C \left(\kappa^2+\kappa^{-2d}\right) \omega_{\rm coef}(2r) \,\norm{u}_{L^\infty(B_{2r}^+)}.
\end{align}

Now, choose $\kappa=\kappa(d, \lambda,\Lambda) \in (0, \frac12)$ such that $2C_0 \kappa \le  \kappa^{1/2}$.
Since \eqref{eq0224sat} holds for any affine function $\ell$, we then obtain
\begin{equation}			\label{eq1110sat}
\varphi(\kappa r) \le \kappa^{1/2} \varphi(r) + C \omega_{\rm coef}(2r)\,\frac{1}{2r} \norm{u}_{L^\infty(B_{2r}^+)},
\end{equation}
where $C=C(d, \lambda, \Lambda)$.
Let $r_0 \in (0, \frac12)$ be a parameter to be determined later.
We will require $r_0 \le r_1$, where $r_1$ is defined in \eqref{eq0818mon}.

By iterating \eqref{eq1110sat}, we obtain, for $j=1,2,\ldots$,
\begin{equation}			\label{eq1656sat}
\varphi(\kappa^j r_0) \le \kappa^{j/2} \varphi(r_0) + C \sum_{i=1}^{j} \kappa^{(i-1)/2} \omega_{\rm coef}(2\kappa^{j-i} r_0)\,\frac{1}{\kappa^{j-i} r_0}\, \norm{u}_{L^\infty(B_{2\kappa^{j-i} r_0}^+)}.
\end{equation}
To handle the sum in \eqref{eq1656sat}, we define, similarly to \cite[(2.15)]{DK17}, a Dini function
\[
\tilde\omega_{\rm coef}(t):= \sum_{i=1}^\infty \kappa^{i/2} \left\{ \omega_{\rm coef}(2\kappa^{-i} t) [ 2\kappa^{-i} t \le 1] +\omega_{\rm coef}(1)[2\kappa^{-i} t >1]\right\},
\]
where Iverson bracket notation is used, i.e., $[P] = 1$ if $P$ is true, and $[P] = 0$ otherwise.
Also, define
\begin{equation}			\label{eq1244sat}
M_j(r_0):=\max_{0\le i < j}\, \frac{1}{2\kappa^i r_0} \,\norm{u}_{L^\infty(B_{2\kappa^i r_0}^+)} \quad \text{for }\;j=1,2, \ldots.
\end{equation}

Then, utilizing \eqref{eq1523sat}, from \eqref{eq1656sat}, we obtain
\begin{equation}			\label{eq4.38}
\varphi(\kappa^j r_0) \le  \frac{\kappa^{j/2}}{r_0} \fint_{B_{r_0}^+} u+ C M_j(r_0)\, \tilde \omega_{\rm coef}(\kappa^j r_0),\quad j=1,2,\ldots.
\end{equation}

\begin{remark}				\label{rmk1147}
We observe that $\omega_{\rm coef}(2t) \lesssim \tilde \omega_{\rm coef}(t)$.
Furthermore, we will use the following facts (see \cite[Lemma 2.7]{DK17}):
\[
\sum_{j=0}^\infty \tilde \omega_{\rm coef}(\kappa^j r) \lesssim \int_0^{r} \frac{\tilde \omega_{\rm coef}(t)}{t}\,dt,\quad \tilde\omega_{\rm coef}(r)  \lesssim \int_0^{r} \frac{\tilde\omega_{\rm coef}(t)}{t}\,dt.
\]
\end{remark}

For each fixed $r$, the infimum in \eqref{eq1100sat} is realized by some $\vec p \in \mathbb{R}^d$ and $q \in \mathbb{R}$.
For $j=0,1,2,\ldots$, let $\vec p_j \in \mathbb{R}^d$ and $q_j \in \mathbb{R}$ be chosen such that
\begin{equation}			\label{eq0203tue}
\varphi(\kappa^j r_0)= \frac{1}{\kappa^j r_0}
\left(\fint_{B_{\kappa^j r_0}^+} \bigabs{u(x)-\vec p_j \cdot x -q_j}^{\frac12}dx\right)^2.
\end{equation}

Next, observe that for $j=0,1,2,\ldots$, we have
\begin{equation}\label{eq0218tue}
\fint_{B_{\kappa^j r_0}^+} \bigabs{\vec p_j \cdot x +q_j}^{\frac12}dx
\le \fint_{B_{\kappa^j r_0}^+} \bigabs{u-\vec p_j \cdot x -q_j}^{\frac12}dx+
\fint_{B_{\kappa^j r_0}^+} \abs{u}^{\frac12}\le 2\fint_{B_{\kappa^j r_0}^+} \abs{u}^{\frac12},
\end{equation}
where we have used \eqref{eq1523sat}.
Furthermore,
\[
\abs{q_j}^{\frac12} = \bigabs{\vec p_j \cdot x + q_j - 2(\vec p_j\cdot x/2 + q_j)}^{\frac12} \le \bigabs{\vec p_j \cdot x + q_j}^{\frac12} + 2^{\frac12}\, \bigabs{\vec p_j \cdot x/2 + q_j}^{\frac12}.
\]
Additionally,
\[
\fint_{B_{\kappa^j r_0}^+} \bigabs{\vec p_j \cdot x/2 + q_j}^{\frac12}dx  = \frac{2^d}{\abs{B_{\kappa^j r_0}^+}} \int_{B_{\kappa^j r_0/2}^+} \bigabs{\vec p_j \cdot x + q_j}^{\frac12} dx \le 2^d\fint_{B_{\kappa^j r_0}^+} \bigabs{\vec p_j\cdot x + q_j}^{\frac12}dx.
\]
Thus,
\begin{equation}		\label{eq0213tue}
\abs{q_j} \le C \left(\fint_{B_{\kappa^j r_0}^+}\bigabs{\vec p_j \cdot x + q_j }^{\frac12}dx\right)^{2}\le C \left(\fint_{B_{\kappa^j r_0}^+} \abs{u}^{\frac12}\right)^{2},\quad j=0,1,2,\ldots.
\end{equation}

Since $u$ is a continuous function vanishing at $0$, \eqref{eq0213tue} immediately implies
\begin{equation}			\label{eq0215tue}
\lim_{j \to \infty} q_j =0.
\end{equation}

\subsection*{Step 4}
We now estimate $\vec p_j$ and $q_j$.
By the quasi-triangle inequality, we have
\[
\bigabs{(\vec p_j - \vec p_{j-1})\cdot x + (q_j -q_{j-1})}^{\frac12} \le \bigabs{u-\vec p_j \cdot x- q_j}^{\frac12} + \bigabs{u-\vec p_{j-1} \cdot x- q_{j-1}}^{\frac12}.
\]
Taking the average over $B_{\kappa^j r_0}^+$ and using the fact that $\bigabs{B_{\kappa^{j-1} r_0}^+}/ \bigabs{B_{\kappa^j r_0}^+} = \kappa^{-d}$, we obtain
\begin{equation}			\label{eq1949sat}
\frac{1}{\kappa^j r_0}
\left(\fint_{B_{\kappa^j r_0}^+} \bigabs{(\vec p_j - \vec p_{j-1})\cdot x + (q_j -q_{j-1})}^{\frac12}dx\right)^2 \le C \varphi(\kappa^j r_0) + C \varphi(\kappa^{j-1} r_0)
\end{equation}
for $j=1,2,\ldots$, where $C=C(d,\kappa)= C(d, \lambda, \Lambda)$.
Next, observe that
\[
\abs{q_j-q_{j-1}}^{\frac12} = \bigabs{(\vec p_j -\vec p_{j-1})\cdot x + (q_j-q_{j-1}) - 2\{(\vec p_j -\vec p_{j-1})\cdot x/2 + (q_j-q_{j-1})\}}^{\frac12}.
\]
Applying the quasi-triangle inequality, we obtain
\[
\abs{q_j-q_{j-1}}^{\frac12} \le \bigabs{(\vec p_j -\vec p_{j-1})\cdot x + (q_j-q_{j-1})}^{\frac12} + 2^{\frac12} \,\bigabs{(\vec p_j -\vec p_{j-1})\cdot x/2 + (q_j-q_{j-1})}^{\frac12}.
\]
Moreover,
\[
\fint_{B_{\kappa^j r_0}} \bigabs{(\vec p_j -\vec p_{j-1})\cdot x/2 + (q_j-q_{j-1})}^{\frac12}dx \leq
2^d\fint_{B_{\kappa^j r_0}} \bigabs{(\vec p_j -\vec p_{j-1})\cdot x + (q_j-q_{j-1})}^{\frac12}dx.
\]
Combining these estimates, we conclude that
\[
\abs{q_j-q_{j-1}} \le C(d) \left(\fint_{B_{\kappa^j r_0}}\bigabs{(\vec p_j - \vec p_{j-1})\cdot x + (q_j -q_{j-1})}^{\frac12}dx\right)^{2},\quad j=1,2,\ldots.
\]
Substituting this into \eqref{eq1949sat}, we derive
\begin{equation}				\label{eq2247sat}
\frac{1}{\kappa^j r_0} \abs{q_j - q_{j-1}} \le C \varphi(\kappa^j r_0)+C \varphi(\kappa^{j-1} r_0),\quad j=1,2,\ldots.
\end{equation}

Note that for any $\vec p \in \mathbb{R}^d$, we have
\begin{equation}		\label{eq1810mon}
\fint_{B_r^+} \abs{\vec p \cdot x}^{\frac12} dx\ge  \abs{\vec p}^{\frac12} \inf_{\abs{\vec e}=1} \fint_{B_r^+}  \abs{\vec e \cdot x}^{\frac12}dx = \abs{\vec p}^{\frac12} r^{\frac12}  \inf_{\abs{\vec e}=1} \fint_{B_1^+} \abs{\vec e \cdot x}^{\frac12}dx= C\abs{\vec p}^{\frac12} r^{\frac12},
\end{equation}
where $C=C(d)>0$.

Then, by using \eqref{eq1810mon}, the quasi-triangle inequality, \eqref{eq1949sat}, \eqref{eq2247sat},  \eqref{eq4.38}, and the observation that $M_{j-1}(r_0) \le M_j(r_0)$, we obtain
\begin{align}
						\nonumber
\abs{\vec p_j-\vec p_{j-1}} &=\frac{C}{\kappa^j r_0} \left(\fint_{B_{\kappa^j r_0}^+}\bigabs{(\vec p_j - \vec p_{j-1})\cdot x}^{\frac12}dx\right)^{2}\\
						\nonumber
& \le \frac{C}{\kappa^j r_0} \left(\fint_{B_{\kappa^j r_0}^+}\bigabs{(\vec p_j - \vec p_{j-1})\cdot x + (q_j -q_{j-1})}^{\frac12}dx\right)^{2} + \frac{C}{\kappa^j r_0} \bigabs{q_j - q_{j-1}}\\
								\label{eq2316sat}
&\le \frac{C\kappa^{j/2}}{r_0}  \fint_{B_{r_0}^+} u+ CM_j(r_0)\left\{\tilde \omega_{\rm coef}(\kappa^j r_0) + \tilde \omega_{\rm coef}(\kappa^{j-1} r_0)\right\}.
\end{align}

To estimate $\abs{\vec p_0}$, we proceed similarly to \eqref{eq2316sat} by applying \eqref{eq0218tue} and \eqref{eq0213tue} with $j=0$, and using H\"older's inequality to obtain
\begin{equation}			\label{eq0230tue}
\abs{\vec p_0} \le \frac{C}{r_0} \fint_{B_{r_0}^+} u,
\end{equation}
where $C=C(d, \lambda, \Lambda)$.

For $k>l\ge 0$, we derive from \eqref{eq2316sat} and the definition of $M_{j}(r_0)$ that
\begin{align}
						\nonumber
\abs{\vec p_k -\vec p_l} \le \sum_{j=l}^{k-1}\, \abs{\vec p_{j+1}-\vec p_j} &\le  \sum_{j=l}^{k-1}  \frac{C\kappa^{(j+1)/2}}{r_0} \fint_{B_{r_0}^+} u+ CM_k(r_0) \sum_{j=l}^{k} \tilde \omega_{\rm coef}(\kappa^j r_0) \\
						\label{eq0946tue}
&\le  \frac{C \kappa^{(l+1)/2}}{(1-\kappa^{1/2})r_0} \fint_{B_{r_0}^+} u+ CM_k(r_0) \int_0^{\kappa^l r_0} \frac{\tilde \omega_{\rm coef}(t)}{t}\,dt,
\end{align}
where we used Remark~\ref{rmk1147}.
In particular, by taking $k=j$ and $l=0$ in \eqref{eq0946tue}, and using \eqref{eq0230tue}, we obtain for $j=1,2,\ldots$ that
\begin{equation}			\label{eq0900tue}
\abs{\vec p_j} \le \abs{\vec p_j-\vec p_0} + \abs{\vec p_0} \le \frac{C}{r_0} \fint_{B_{r_0}^+} u+ CM_j(r_0) \int_0^{r_0} \frac{\tilde \omega_{\rm coef}(t)}{t}\,dt.
\end{equation}

Similarly, we obtain from \eqref{eq2247sat} that for $k>l\ge 0$, we have
\begin{equation}			\label{eq0947tue}
\abs{q_k - q_l} \le C \frac{\kappa^{3(l+1)/2}}{1-\kappa^{3/2}} \fint_{B_{r_0}^+} u+ C \kappa^l r_0 M_k(r_0) \int_0^{\kappa^l r_0} \frac{\tilde \omega_{\rm coef}(t)}{t}\,dt.
\end{equation}

\subsection*{Step 5}
We will derive improved estimates for $q_j$ using the following lemma, where we set
\[
\tilde{u}(x):=u(x)-\vec p_j\cdot x-q_j.
\]
We note that $\tilde{u}$ satisfies
\begin{align*}
L^{*} \tilde{u}=-\dv^2 (\mathbf{A} (\vec p_j \cdot x + q_j))+\dv((\vec p_j \cdot x + q_j) \vec b) -(\vec p_j \cdot x + q_j)c&\quad\text{in }\; B_2^{+}\\
(\mathbf{A} \nu \cdot \nu)\tilde{u}= -\mathbf{A} (\vec p_j \cdot x + q_j) \nu \cdot \nu &\quad\text{on }\; T_2.
\end{align*}

\begin{lemma}			\label{lem0921tue}
For $0<r \le \frac12$, we have
\[
\sup_{B_r^+}\, \abs{u-\vec p_j\cdot x-q_j} \le C \left\{ \left( \fint_{B_{2r}^+} \abs{u-\vec p_j\cdot x-q_j}^{\frac12} \right)^{2} +(r\abs{\vec p_j}+\abs{q_j}) \int_0^{2r} \frac{\omega_{\rm coef}(t)}{t}\,dt \right\},
\]
where $C=C(d, \lambda, \Lambda, \omega_{\rm coef})$.
\end{lemma}

\begin{proof}
In the case when the lower order terms $\vec b$ and $c$ are not present, the result follows from \cite[Lemma 2.3]{KL21}, which itself refines the proof of \cite[Lemma 2.26]{DEK18}.
As in \cite{KL21}, our goal is to control the following quantity for $x_0 \in B_{3r/2}^+$ and $0<t\le r/4$:
\[
\phi(x_0, t):= \inf_{q \in \mathbb{R}} \left(\fint_{B_t(x_0) \cap B_{2r}^+} \abs{\tilde{u}-q}^{\frac12} \right)^2,\quad\mbox{where }\; \tilde u=u-\vec p_j\cdot x-q_j.
\]

We consider two cases: $B_t(x_0) \cap \partial\mathbb{R}^d_+ = \emptyset$ and $B_t(x_0) \cap \partial\mathbb{R}^d_+ \neq \emptyset$.

\smallskip
\noindent
\textit{Case 1:} $B_t(x_0) \cap \partial\mathbb{R}^d_+ = \emptyset$.
\smallskip

In this case, we have $B_t(x_0) \subset B_{2r}^+$, and we closely follow the proof of \cite[Lemma 2.2]{KL21}.
For $x_0 \in B_{3r/2}$ and $0<t \le r/4$, let $\bar{\mathbf A}$ denote the average of $\mathbf{A}$ over the ball $B_t(x_0)$.
Define
\begin{equation}			\label{eq0721sun}
\mathbf f= (\bar{\mathbf A}-\mathbf{A}) (\vec p_j \cdot x + q_j),\quad
\vec g=(\vec p_j \cdot x + q_j) \vec b,\quad
 h=-(\vec p_j \cdot x + q_j)c.
\end{equation}
We decompose $\tilde{u}$ as $\tilde{u}=v_1+v_2$, where $v_1 \in L^p(B_t(x_0))$ for some $p>1$ is the solution of the problem:
\[
\left\{
\begin{aligned}
\dv^2(\bar{\mathbf A}v_1) &= \dv^2 \left(\mathbf{f}-(\mathbf{A}- \bar{\mathbf A})\tilde{u}\right)+\dv(\vec g+\vec b \tilde{u})+h-c\tilde{u}\;\mbox{ in }\;B_t(x_0),\\
(\bar{\mathbf A}\nu\cdot \nu)v_1&=\left(\mathbf{f}-(\mathbf{A}- \bar{\mathbf A})\tilde{u}\right)\nu\cdot \nu \;\mbox{ on }\;\partial B_t(x_0).
\end{aligned}
\right.
\]
Then, by Lemma~\ref{lem01} via rescaling, we obtain the following estimate:
\begin{multline}				\label{eq2025mon}
\left(\fint_{B_t(x_0)} \abs{v_1}^{\frac12} \right)^{2} \lesssim \fint_{B_t(x_0)}\abs{\mathbf f}+ \left(\fint_{B_t(x_0)}\abs{\mathbf{A}-\bar{\mathbf A}}\right)\norm{\tilde{u}}_{L^\infty(B_t(x_0))}+t \fint_{B_t(x_0)} \abs{\vec g}\\
+ \left(t \fint_{B_t(x_0)}\abs{\vec b}\right)\norm{\tilde{u}}_{L^\infty(B_t(x_0))} + t^2 \fint_{B_t(x_0)} \abs{h} +  \left(t^2 \fint_{B_t(x_0)}\abs{c}\right)\norm{\tilde{u}}_{L^\infty(B_t(x_0))}.
\end{multline}
By definitions of $\mathbf f$, $\vec g$, and $h$, we obtain
\[
\fint_{B_t(x_0)}\abs{\mathbf f} +t\fint_{B_t(x_0)}\abs{\vec g} + t^2\fint_{B_t(x)}\abs{h} \le (2r \abs{\vec p_j}+\abs{q_j}) \left( \fint_{B_t(x_0)}\abs{\mathbf{A}-\bar{\mathbf A}}+t \fint_{B_t(x_0)}\abs{\vec b}+t^2 \fint_{B_t(x_0)}\abs{c} \right).
\]
Thus, using the definition \eqref{omega_coef}, we deduce from \eqref{eq2025mon} that
\begin{equation}				\label{eq2026mon}
\left(\fint_{B_t(x_0)} \abs{v_1}^{\frac12} \right)^{2} \le C \omega_{\rm coef}(t)\norm{\tilde{u}}_{L^\infty(B_t(x_0))} + C(r\abs{\vec a_j}+\abs{b_j}) \omega_{\rm coef}(t),
\end{equation}
where $C=C(d, \lambda, \Lambda)$.
We note that \eqref{eq2026mon} corresponds to (A.1) in \cite{KL21}.

On the other hand, note that $v_2=\tilde{u}-v_1$ satisfies
\[
\dv^2(\bar{\mathbf A}v_2) = 0 \;\;\text{ in }\;  B_t(x_0).
\]
Thus, $v_2$ satisfies the estimate corresponding to (A.2) in \cite{KL21}.
The remainder of the proof follows the same arguments as in \cite[Lemma 2.2]{KL21}.

\smallskip
\noindent
\textit{Case 2:} $B_t(x_0) \cap \partial\mathbb{R}^d_+ \neq \emptyset$.
\smallskip

In this case, we replace $B_t(x_0) \cap B_{2r}^+$ with $B_{2t}^{+}(\bar x_0)$, where $\bar x_0$ is the projection of $x_0$ onto $\partial \mathbb{R}^d_+$, and replicate the argument from \cite{DEK18}.

For $x_0 \in T_{3r/2}$ and $0<t \le r/4$, let $\bar{\mathbf A}$ denote the averages of $\mathbf{A}$ over the half-ball $B_{2t}^+(x_0)$.
Let $\mathbf f$, $\vec g$, and $h$ be defined as in \eqref{eq0721sun}, and decompose $\tilde u$ as $\tilde u=v_1+v_2$, where $v_1 \in L^p(\mathcal{D}_t(x_0))$ for some $p>1$ is the solution of the following problem:
\[
\left\{
\begin{aligned}
\dv^2(\bar{\mathbf A}v_1) &= \dv^2 \left(\mathbf{f}-(\mathbf{A}- \bar{\mathbf A})\tilde{u}\right)+\dv(\vec g+\vec b \tilde{u})+h-c\tilde{u}\;\mbox{ in }\;\mathcal{D}_t(x_0),\\
(\bar{\mathbf A}\nu\cdot \nu)v_1&=\left(\mathbf{f}-(\mathbf{A}- \bar{\mathbf A}) \tilde{u}\right)\nu\cdot \nu \;\mbox{ on }\;\partial \mathcal{D}_t(x_0).
\end{aligned}
\right.
\]
Then, similar to \eqref{eq2025mon}, we obtain the following estimate:
\begin{multline}				\label{eq1458mon}
\left(\fint_{B_t^+(x_0)} \abs{v_1}^{\frac12} \right)^{2} \lesssim \fint_{B_{2t}^+(x_0)}\abs{\mathbf f}+ \left(\fint_{B_{2t}^+(x_0)}\abs{\mathbf{A}-\bar{\mathbf A}}\right)\norm{\tilde{u}}_{L^\infty(B_{2t}^+(x_0))}
+ t \fint_{B_{2t}^+(x_0)} \abs{\vec g}\\
+\left(t \fint_{B_{2t}^+(x_0)}\abs{\vec b}\right)\norm{\tilde{u}}_{L^\infty(B_{2t}^+(x_0))} + t^2 \fint_{B_{2t}^+(x_0)} \abs{h} +  \left(t^2 \fint_{B_{2t}^+(x_0)}\abs{c}\right)\norm{\tilde{u}}_{L^\infty(B_{2t}^+(x_0))},
\end{multline}
where we used $B_{t}^+(x_0) \subset \mathcal{D}_t(x_0) \subset B_{2t}^+(x_0)$.

Thus, from \eqref{eq1458mon}, similarly to \eqref{eq2026mon}, we obtain the following:
\[
\left(\fint_{B^+_t(x_0)} \abs{v_1}^{\frac12} \right)^{2} \le C \omega_{\rm coef}(2t)\,\norm{\tilde{u}}_{L^\infty(B_t^{+}(x_0))}+ C(r\abs{\vec p_j}+\abs{q_j}) \omega_{\rm coef}(2t),
\]
where $C=C(d, \lambda,\Lambda)$.

On the other hand, note that $v_2=\tilde{u}-v_1$ satisfies
\[
\dv^2(\bar{\mathbf A}v_2) =0 \;\text{ in }\; B^+_t(x_0),\quad v_2=\bar{\mathbf A}(\vec p_j \cdot x + q_j) \nu\cdot \nu\;\text{ on }\;T_t(x_0).
\]
Applying \cite[Lemma 4.6]{CDKK24} with scaling, we obtain
\[
\norm{D^2 v_2}_{L^\infty(B_{t/2}^+(x_0))}\le C t^{-2-2d} \norm{v_2}_{L^{1/2}(B_{t}^+(x_0))}.
\]
Combining this with the interpolation inequality yields
\[
\norm{D v_2}_{L^\infty(B_{t/2}^+(x_0))}\le C t^{-1-2d}\norm{v_2}_{L^{1/2}(B_{t}^+(x_0))}.
\]
The above estimates remain valid if $v_2$ is replaced by $v_2 - \ell$ for any affine function $\ell$.
In particular, we have
\[
\norm{D v_2}_{L^\infty(B_{t/2}^+(x_0))}\le C t^{-1} \norm{v_2-q}_{L^{1/2}(B_{t}^+(x_0))},\quad \forall q \in \mathbb{R}.
\]
The remainder of the proof involves a standard modification of the argument in \cite[Lemma 2.26]{DEK18}, which we leave to the reader.
\end{proof}

By \eqref{eq0203tue} and \eqref{eq4.38}, we obtain
\begin{equation}				\label{eq1150sun}
\left(\fint_{B_{\kappa^j r_0}^+} \bigabs{u(x)-\vec p_j \cdot x -q_j}^{\frac12}dx\right)^2 \le  \kappa^{3j/2} \fint_{B_{r_0}^+} u+ C M_j(r_0)\, \kappa^j r_0 \tilde \omega_{\rm coef}(\kappa^j r_0).
\end{equation}
Then, applying Lemma \ref{lem0921tue} and \eqref{eq1150sun}, we obtain
\begin{multline}				\label{eq0934wed0}
 \norm{u-\vec p_j\cdot x-q_j}_{L^\infty(B_{\kappa^j r_0/2}^+)} \le C\kappa^{3j/2} \fint_{B_{r_0}^+} u + C\kappa^j r_0 M_j(r_0) \tilde\omega_{\rm coef}(\kappa^j r_0)\\
+ C \left(\kappa^j r_0\abs{\vec p_j}+\abs{q_j}\right) \int_0^{\kappa^j r_0} \frac{\omega_{\rm coef}(t)}{t}\,dt,\quad j=1,2,\ldots,
\end{multline}
where $C=C(d, \lambda, \Lambda, \omega_{\rm coef})$.

In particular, taking $x=0$ in \eqref{eq0934wed0}, we infer that
\begin{multline*}
\abs{q_j} \le C\kappa^{3j/2} \fint_{B_{r_0}^+} u + C\kappa^j r_0 M_j(r_0) \tilde\omega_{\rm coef}(\kappa^j r_0) +C \abs{\vec p_j}\kappa^j r_0 \int_0^{\kappa^j r_0} \frac{\omega_{\rm coef}(t)}{t}\,dt\\
+C \abs{q_j} \int_0^{\kappa^j r_0} \frac{\omega_{\rm coef}(t)}{t}\,dt.
\end{multline*}
Let us fix $r_2>0$ such that
\begin{equation}			\label{eq0903tue}
C \int_0^{r_2} \frac{\omega_{\rm coef}(t)}{t}\,dt\le \frac12.
\end{equation}
Note that $r_2$ depends solely on $d$, $\lambda$, $\Lambda$, and $\omega_{\rm coef}$.

At this stage, we have not yet chosen $r_0 \in (0,\frac12]$.
We will require $r_0 \le r_2$ so that we obtain
\[
\abs{q_j} \le  C\kappa^{3j/2} \fint_{B_{r_0}^+} u + C\kappa^j r_0 M_j(r_0) \tilde\omega_{\rm coef}(\kappa^j r_0)
+C\abs{\vec p_j}\kappa^j r_0\int_0^{\kappa^j r_0} \frac{\omega_{\rm coef}(t)}{t}\,dt.
\]
Together with \eqref{eq0900tue} and Remark~\ref{rmk1147}, this leads to
\begin{equation}			\label{eq7.51}
\abs{q_j} \le C \kappa^{j} r_0 \left\{\kappa^{j/2}+\int_0^{\kappa^j r_0} \frac{\omega_{\rm coef}(t)}{t}\,dt\right\}\frac{1}{r_0} \fint_{B_{r_0}^+}u +C\kappa^j r_0 M_j(r_0) \int_0^{\kappa^j r_0} \frac{\tilde \omega_{\rm coef}(t)}{t}\,dt.
\end{equation}

\subsection*{Step 6}
We now proceed to demonstrate that the sequences $\{\vec p_j\}$ and $\{q_j\}$ are Cauchy and therefore converge.
Using \eqref{eq0934wed0}, \eqref{eq0900tue}, \eqref{eq7.51}, \eqref{eq0903tue}, and Remark~\ref{rmk1147}, we obtain
\begin{multline}				\label{eq1920thu}
\norm{u-\vec p_j\cdot x-q_j}_{L^\infty(B_{\kappa^j r_0/2}^+)}
\le C \kappa^j r_0 \left\{\kappa^{j/2}+ \int_0^{\kappa^j r_0} \frac{\omega_{\rm coef}(t)}{t}\,dt \right\}\frac{1}{r_0} \fint_{B_{r_0}^+} u\\
+C\kappa^j r_0 M_j(r_0)  \int_0^{\kappa^j r_0} \frac{\tilde \omega_{\rm coef}(t)}{t}\,dt.
\end{multline}
Then, from \eqref{eq1920thu}, \eqref{eq0900tue}, \eqref{eq7.51}, and Remark~\ref{rmk1147}, we infer that
\begin{equation}				\label{eq1555sun}
\frac {1}{\kappa^{j} r_0}\norm{u}_{L^\infty(B_{\kappa^j r_0/2}^+)}
\le \frac C {r_0} \fint_{B_{r_0}^+} u + C M_j(r_0) \int_0^{r_0} \frac{\tilde \omega_{\rm coef}(t)}{t}\,dt, \quad j=1,2,\ldots,
\end{equation}
where $C=C(d, \lambda, \Lambda, \omega_{\rm coef})$.

\begin{lemma}				\label{lem1548sun}
There exists $r_3 =r_3(d, \lambda, \Lambda, \omega_{\rm coef}) \in (0,\frac12]$ such that
\[
\sup_{j \ge 1} M_j(r_0) =\sup_{i \ge 0} \frac{1}{2\kappa^i r_0} \,\norm{u}_{L^\infty(B_{2\kappa^i r_0}^+)} \le \frac{C}{r_0} \fint_{B_{4r_0}^+} u,\quad \forall r_0 \in (0,r_3]
\]
where $C=C(d, \lambda, \Lambda, \omega_{\rm coef})$.
\end{lemma}

\begin{proof}
We denote
\[
\quad c_i= \frac{1}{2\kappa^i r_0} \norm{u}_{L^\infty(B_{2\kappa^i r_0}^+)},\quad i=0,1,2,\ldots.
\]
Then, by \eqref{eq1244sat}, we clearly have
\begin{equation}			\label{eq0906tue}
M_1(r_0)=c_0\quad\text{and}\quad
M_{j+1}(r_0)=\max(M_j(r_0), c_j),\quad j=1,2,\ldots.
\end{equation}
Recall that we have already chosen $\kappa=\kappa(d, \lambda, \Lambda) \in (0,\frac12)$, and thus, from \eqref{eq1555sun}, we conclude that there exists a constant $C=C(d, \lambda, \Lambda, \omega_{\rm coef})>0$ such that
\begin{equation}			\label{eq0848tue}
c_{j+1} \le C \left\{\frac{1}{r_0} \fint_{B_{r_0}^+}u+ M_j(r_0) \int_0^{r_0} \frac{\tilde \omega_{\rm coef}(t)}{t}\,dt \right\},\quad j=1,2,\ldots.
\end{equation}
By \cite[Theorem 1.10]{DEK18}, we have
\begin{equation}			\label{eq0849tue}
c_1 =\frac{1}{2\kappa r_0} \norm{u}_{L^\infty(B_{2\kappa r_0}^+)} \le \frac{1}{2\kappa r_0} \norm{u}_{L^\infty(B_{2r_0}^+)} =\frac{1}{\kappa} c_0  \le \frac{C}{r_0} \fint_{B_{4r_0}^+} u,
\end{equation}
where $C=C(d, \lambda, \Lambda, \omega_{\rm coef})$.

From \eqref{eq0848tue} and \eqref{eq0849tue}, we see that there is $\gamma=\gamma(d, \lambda, \Lambda,\omega_{\rm coef})>0$ such that
\[
c_0,\, c_1 \le \frac{\gamma}{r_0} \fint_{B_{4r_0}^+} u \; \text{ and }\; c_{j+1} \le \gamma \left\{\frac{1}{r_0} \fint_{B_{4r_0}^+}u+ M_j(r_0) \int_0^{r_0} \frac{\tilde \omega_{\rm coef}(t)}{t}\,dt \right\},\;\; j=1,2,\ldots.
\]

Now, we fix a number $\rho_0 \in (0,\frac12]$ such that
\[
\gamma \int_0^{\rho_0} \frac{\tilde \omega_{\rm coef}(t)}{t}\,dt \le \frac12
\]
and also $\rho_0 \le r_2$, where $r_2$ is in \eqref{eq0903tue}.
Then, if $r_0 \le \rho_0$, we have
\begin{equation}			\label{eq2114mon}
c_0, c_1 \le \frac{\gamma}{r_0} \fint_{B_{4r_0}^+}u \; \text{ and }\; c_{j+1} \le \frac{\gamma}{r_0} \fint_{B_{4r_0}^+}\abs{u} +\frac12 M_j(r_0),\;\; j=1,2,\ldots.
\end{equation}

By induction, it follows form \eqref{eq0906tue} and \eqref{eq2114mon} that
\[
c_{2k},\,c_{2k+1},\,M_{2k+1}(r_0),\,M_{2k+2}(r_0) \le \frac{\gamma}{r_0} \fint_{B_{4r_0}^+}u \,\left( \sum_{i=0}^{k} \frac{1}{2^i}\right),\quad k=0,1,2,\ldots,
\]
and the lemma follows.
\end{proof}

\begin{remark}			\label{rmk1203tue}
The proof of Lemma \ref{lem1548sun} actually shows that for $0<\rho \le r_0$, we have
\[
\sup_{0<r\le \rho} \,\frac{1}{r}  \norm{u}_{L^\infty(B_r^+)} \le \frac{C}{\rho} \fint_{B_{4\rho}^+} \abs{u},
\]
where $C=C(d, \lambda, \Lambda, \omega_{\rm coef})$.
To see this, we repeat the same argument with $\rho$ in place of $r_0$, and use the following inequality:
\[
\frac{1}{\kappa^{j}\rho} \norm{u}_{L^\infty(B_{\kappa^{j+1} \rho}^+)} \le \sup_{\kappa^{j+1}\rho \le r \le \kappa^j \rho}  \frac{1}{r} \norm{u}_{L^\infty(B_r^+)} \le \frac{1}{\kappa^{j+1}\rho} \norm{u}_{L^\infty(B_{\kappa^j \rho}^+)}.
\]
\end{remark}

Hereafter, we will further impose the condition that $r_0 \le r_3$, while still retaining the freedom to choose $r_0$.
Now, Lemma \ref{lem1548sun} and \eqref{eq0946tue} imply that the sequence $\set{\vec p_j}$ is a Cauchy sequence in $\mathbb{R}^d$, and thus $\vec p_j \to \hat{\vec p}$ for some $\hat{\vec p} \in \mathbb{R}^d$.
Moreover, by taking the limit as $k\to \infty$ in \eqref{eq0946tue} and \eqref{eq0947tue} (while recalling Lemma \ref{lem1548sun} and \eqref{eq0215tue}), respectively, and then setting $l=j$, we obtain the following estimates:
\begin{equation}	\label{eq1806sun}
\begin{aligned}
\abs{\vec p_j-\hat{\vec p}} &\le C \left\{ \kappa^{j/2}+ \int_0^{\kappa^j r_0} \frac{\tilde \omega_{\rm coef}(t)}{t}\,dt\right\}\frac{1}{r_0} \fint_{B_{4r_0}^+}u,\\
\abs{q_j}  &\le C \kappa^j \left\{\kappa^{j/2} +\int_0^{\kappa^j r_0} \frac{\tilde \omega_{\rm coef}(t)}{t}\,dt \right\}\fint_{B_{4r_0}^+} u,
\end{aligned}
\end{equation}
where $C=C(d, \lambda, \Lambda, \omega_{\rm coef})$.
It is worth noting that the constants in \eqref{eq1806sun} remain bounded as $r_0 \to 0$.

By the triangle inequality, \eqref{eq1920thu}, \eqref{eq1806sun}, and Remark~\ref{rmk1147}, we obtain
\begin{align}			\nonumber
\norm{u-\hat{\vec p}\cdot x}_{L^\infty(B_{\kappa^j r_0/2}^+)}& \le \norm{u- \vec p_j \cdot x - q_j}_{L^\infty(B_{\kappa^j r_0/2}^+)} + \frac{\kappa^j r_0}{2} \abs{\vec p_j - \hat{\vec p}} +  \abs{q_j}\\
					\label{eq2221sun}
&\le C \kappa^j r_0\left\{\kappa^{j/2}+\int_0^{\kappa^j r_0} \frac{\tilde\omega_{\rm coef}(t)}{t}\,dt\right\} \frac{1}{r_0} \fint_{B_{4r_0}^+} u.
\end{align}

\subsection*{Step 7}
We will establish a lower bound for $\abs{\hat{\vec p}}$.
It follows from \eqref{eq2221sun} that
\begin{equation}		\label{eq1036wed}
\frac{1}{r} \norm{u-\hat{\vec p}\cdot x}_{L^\infty(B_r^+)}  \le \varrho_{\rm coef}(r) \left( \frac{1}{r_0} \fint_{B_{4r_0}^+} u \right),
\end{equation}
where
\[
\varrho_{\rm coef}(r)=C\left\{\left(\frac{2r}{\kappa r_0}\right)^{\frac12}+\int_0^{2r/\kappa} \frac{\tilde\omega_{\rm coef}(t)}{t}\,dt\right\}
\]
is a modulus of continuity.

In particular, from \eqref{eq1036wed}, we conclude that $u$ is differentiable at $0$ and that $Du(0)=\hat{\vec p}$.
Since $u=0$ on $\partial B_2^+\cap \set{x_d=0}$, it follows that
\[
Du(0)=\hat{\vec p}=(\hat{\vec p} \cdot \vec e_d)\vec e_d.
\]
Consequently, we obtain
\[
\left(\fint_{B_r^+} \abs{\hat{\vec p} \cdot x}^{\frac12}\right)^2=  \abs{\hat{\vec p}} \left(\fint_{B_r^+}  \abs{x_d}^{\frac12}\right)^2.
\]
Thus, we arrive at
\begin{equation}			\label{eq1014sat}
\abs{\hat{\vec p}} = \frac{\mu_0}{r} \left(\fint_{B_r^+} \abs{\hat{\vec p} \cdot x}^{\frac12} \right)^2,\quad\text{where }\;1/\mu_0:=\left(\fint_{B_1^+} \abs{x_d}^{\frac12}\right)^{2}>0.
\end{equation}

Using the quasi-triangle inequality \eqref{qt}, we obtain
\[
\left(\fint_{B_{\kappa^j r_0}^+} \abs{\hat{\vec p}\cdot x}^{\frac12}\right)^{2} \ge \frac12 \left(\fint_{B_{\kappa^j r_0}^+} \abs{u}^{\frac12}\right)^{2} - \left(\fint_{B_{\kappa^j r_0}^+} \abs{u-\hat{\vec p}\cdot x}^{\frac12}\right)^{2}.
\]
Furthermore, we have
\begin{align*}
\left(\fint_{B_{\kappa^j r_0}^+} \abs{u-\hat{\vec p}\cdot x}^{\frac12}\right)^{2} &\le 
2\left(\fint_{B_{\kappa^j r_0}^+} \abs{u- \vec p_j \cdot x -q_j}^{\frac12}\right)^{2} +
2\left(\fint_{B_{\kappa^j r_0}^+} \abs{(\vec p_j-\hat{\vec p}) \cdot x+q_j}^{\frac12}\right)^{2}\\
&\le 2\left(\fint_{B_{\kappa^j r_0}^+} \abs{u- \vec p_j \cdot x -q_j}^{\frac12}\right)^{2} + 8 \kappa^j r_0 \abs{\vec p_j-\hat{\vec p}} + 4 \abs{q_j}. 
\end{align*}
By \eqref{eq1150sun} and Lemma \ref{lem1548sun}, it follows that
\[
\left(\fint_{B_{\kappa^j r_0}^+} \bigabs{u(x)-\vec p_j \cdot x -q_j}^{\frac12}dx\right)^2 \le\kappa^{j}r_0 \left( \kappa^{j/2} +\tilde \omega_{\rm coef}(\kappa^j r_0) \right) \frac{1}{r_0} \fint_{B_{4r_0}^+} u.
\]
Combining the above estimates with \eqref{eq1806sun}, we deduce from \eqref{eq1014sat} that
\[
\abs{\hat{\vec p}} \ge \frac{\gamma_0}{2\kappa^j r_0}
\left(\fint_{B_{\kappa^j r_0}^+} u^{\frac12}\right)^{2} -C \left( \kappa^{j/2}+\int_0^{\kappa^j r_0} \frac{\tilde \omega_{\rm coef}(t)}{t}\,dt \right) \frac{1}{r_0} \fint_{B_{4r_0}^+} u,
\]
where we used Remark \ref{rmk1147}, and  the constant $C$ depends on $d$, $\lambda$, $\Lambda$, and $\omega_{\rm coef}$, i.e., $C=C(d, \lambda, \Lambda,\omega_{\rm coef})$.

Using Lemma \ref{lem_dp}, we further deduce
\begin{equation}			\label{eq1528sat}
\abs{\hat{\vec p}} \ge \frac{\mu_0}{4\kappa^j r_0}
\left(\fint_{B_{\kappa^j r_0}^+} u^{\frac12}\right)^{2} -C \left( \kappa^{j/2}+\int_0^{\kappa^j r_0} \frac{\tilde \omega_{\rm coef}(t)}{t}\,dt \right) \frac{1}{r_0} \fint_{B_{r_0}^+} u,
\end{equation}
where $C=C(d, \lambda, \Lambda,\omega_{\rm coef})$.

Now, we take $N$ to be the smallest integer such that in \eqref{eq1528sat}, we have
\begin{equation}			\label{eq0630mon}
C \left( \kappa^{N/2}+\int_0^{\kappa^N} \frac{\tilde \omega_{\rm coef}(t)}{t}\,dt \right) \le \frac{\nu_0 \beta_0}{128}.
\end{equation}
We note that $N$ depends only on $d$, $\lambda$, $\Lambda$, and $\omega_{\rm coef}$.

Recall that we assume $r_0 \le \frac12$. 
Hence, we derive from \eqref{eq1528sat} and \eqref{eq0630mon} that
\begin{equation}			\label{eq1538sat}
\abs{\hat{\vec p}} \ge \frac{\mu_0}{4\kappa^N r_0}
\left(\fint_{B_{\kappa^N r_0}^+} u^{\frac12}\right)^{2} - \frac{\nu_0 \beta_0}{128} \frac{1}{r_0} \fint_{B_{r_0}^+} u.
\end{equation}

We will derive the following estimate:
\begin{equation}			\label{eq1222sun}
\frac{1}{\kappa^N r_0}\left(\fint_{B_{\kappa^N r_0}^+} u^{\frac12}\right)^{2} \ge  \frac{\nu_0 \beta_0}{16\mu_0} \frac{1}{r_0} \fint_{B_{r_0}^+} u.
\end{equation}
From \eqref{eq1538sat} and \eqref{eq1222sun}, it follows that
\[
D_d u(0)=\abs{\hat{\vec p}} \ge \frac{\nu_0 \beta_0}{128}
 \frac{1}{r_0} \fint_{B_{r_0}^+} u,
\]
which confirms the desired estimate \eqref{eq2017sun}.

It remains to establish \eqref{eq1222sun} to complete the proof. By \eqref{eq1414tue}, \eqref{eq1556wed}, and \eqref{eq1014sat}, for any $\epsilon \in (0,\frac12)$, we obtain
\begin{equation}			\label{eq1216fri}
\frac{1}{\epsilon r_0}\left(\fint_{B_{\epsilon r_0}^+} v^{\frac12} \right)^2 \ge \frac{1}{\epsilon r_0} \frac{\beta_0}{r_0}\left(\frac{\nu_0}{4} \fint_{B_{r_0}^+} u \right)\left(\fint_{B_{\epsilon r_0}^+} \abs{x_d}^{\frac12} \right)^2 = \frac{\nu_0 \beta_0}{4\mu_0} \left(\frac{1}{r_0} \fint_{B_{r_0}^+} u \right).
\end{equation}
Furthermore, similar to \eqref{eq1556wed}, we have
\begin{equation}			\label{eq1110fri}
\left(\fint_{B_{\epsilon r_0}^+} u^{\frac12} \right)^2 \ge \frac12 \left(\fint_{B_{\epsilon r_0}^+} v^{\frac12} \right)^2 -\left(\fint_{B_{\epsilon r_0}^+} \abs{w}^{\frac12} \right)^2.
\end{equation}

On the other hand, from \eqref{eq2050sat}, we derive
\begin{equation}			\label{eq1213fri}
\frac{1}{\epsilon r_0}\left(\fint_{B_{\epsilon r_0}^+} \abs{w}^{\frac12} \right)^2  \le \frac{1}{\epsilon r_0}\frac{1}{\epsilon^{2d}}\left(\fint_{B_{r_0}^+} \abs{w}^{\frac12} \right)^2  \le \frac{C \omega_{\rm coef}(2r_0) }{\epsilon^{(2d+1)}} \left(\frac{1}{r_0}\fint_{B_{r_0}^+} u\right).
\end{equation}
Setting $\epsilon=\kappa^N$ in \eqref{eq1213fri} and selecting $r_4>0$ such that
\[
C \kappa^{-(2d+1)N} \omega_{\rm coef}(2r_4) \le \frac{\nu_0 \beta_0}{16\mu_0},
\]
we take $r_0 =\min(r_1, r_2, r_3, r_4, \frac12)$.
Combining \eqref{eq1110fri}, \eqref{eq1216fri}, and \eqref{eq1213fri}, we obtain
\[
\frac{1}{\kappa^N r_0}\left(\fint_{B_{\kappa^N r_0}^+} u^{\frac12} \right)^2 \ge \frac{\nu_0 \beta_0}{16\mu_0} \left(\frac{1}{r_0} \fint_{B_{r_0}^+} u \right).
\]
This completes the proof for \eqref{eq1222sun}.

\subsection*{Step 8}
%
To verify \eqref{eq_bdry_hnk}, let $\bar x$ denote its projection onto $\partial \bR^d_+$, i.e., $\bar x=(x_1,\ldots, x_{d-1},0)$.
For for $x \in B_{r_0/2}^+$, by the estimate \eqref{eq1036wed} applied at $\bar x$ in place of $0$, and the fact that $\hat{\vec p}=Du(\bar x)=D_du(\bar x )\vec e_d$, we obtain
\[
\abs{u(x)- D_du(\bar x)x_d}  \le C x_d \,\varrho_{\rm coef}(x_d) \left( \frac{1}{r_0} \fint_{B_{4r_0}^+(\bar x)} u \right) \le C x_d \, \varrho_{\rm coef}(r_0) \left( \frac{1}{r_0} \fint_{B_{5r_0}^+} u \right).
\]
Note that \eqref{eq2017sun} applied at $\bar x$ in place of $0$, we have
\[
D_d u (\bar x) \ge \frac{C}{r_0} \fint_{B_{r_0}^+(\bar x)} u \ge \frac{C}{r_0} \fint_{B_{r_0/2}^+} u,
\]
where we used $B_{r_0/2}^+ \subset B_{r_0}^+(\bar x)$ when $x \in B_{r_0/2}^+$.

Combining the above estimates via the triangle inequality and applying the doubling property of $u$ from Lemma \ref{lem_dp}, we obtain
\[
\frac{u(x)}{x_d} \ge \frac{C}{r_0} \fint_{B_{r_0/2}^+} u - C \varrho_{\rm coef}(r_0) \left( \frac{1}{r_0} \fint_{B_{5r_0}^+} u \right) \ge C(1-\varrho_{\rm coef}(r_0))\, \frac{1}{r_0} \fint_{B_{r_0/2}^+} u.
\]
By choosing $\tilde r_0$ such that $\varrho_{\rm coef}(2\tilde r_0) \le \frac12$, and ensuring that it satisfies all the previous conditions on $r_0$, namely, $\tilde r_0 \le \min(r_1, r_2, r_3, r_4,\frac12)$, we obtain \eqref{eq_bdry_hnk} by relabeling $r=r_0/2$.
\qed

\section{Proof of Theorem \ref{GreenEstimate}}		\label{sec3}
The existence of the Green's function $G(x,y)$ for the non-divergence form operator $L$ in the domain $\Omega$ is established in \cite[Theorem 7.3]{DKK25}, where it is shown that
\begin{equation}			\label{eq0900thu}
0 \le G(x,y) \le C_0 \abs{x-y}^{2-d},\quad x\neq y \in \Omega,
\end{equation}
where $C_0>0$ is a constant depending only on $d$, $\lambda$, $\Lambda$, $\omega_{\mathbf A}$, $p_0$, $\norm{\vec b}_{L^{p_0}}$, $\norm{c}_{L^{p_0/2}}$, and $\diam \Omega$.
We also note that the Green's function for the double divergence form operator $L^*$ satisfies the following symmetry relation:
\begin{equation}			\label{eq2005tue}
G(x,y)=G^*(y,x)\quad x\neq y \in \Omega.
\end{equation}

\subsection*{The upper bound}
We aim to establish the following upper bound for $G(x,y)$:
\begin{equation}		\label{eq0943wed}
G(x,y) \le \frac{C}{\abs{x-y}^{d-2}} \left(1 \wedge \frac{\mathtt{d}(y)}{\abs{x-y}}\right) \left(1 \wedge \frac{\mathtt{d}(x)}{\abs{x-y}}\right),\quad x\neq y.
\end{equation}

This estimate follows directly from the argument used in the proof of \cite[Theorem 5.4]{CDKK24}.
For completeness, we will reproduce the key steps here.

We begin by deriving the intermediate bound
\begin{equation}		\label{eq0901thu}
G(x,y) \le \frac{C}{\abs{x-y}^{d-2}} \left(1 \wedge \frac{\mathtt{d}(y)}{\abs{x-y}}\right),\quad x\neq y.
\end{equation}
It suffices to consider the case when $\mathtt{d}(y) \le \frac14 \abs{x-y}$, as otherwise, \eqref{eq0900thu} immediately yields \eqref{eq0901thu} with $C=4C_0$.
Let us define
\[
R:=\abs{x-y}
\]
and choose $y_0 \in \partial\Omega$ such that $\abs{y-y_0}=\mathtt{d}(y)$.
By the triangle inequality, for any $z \in \Omega \cap B_{R/2}(y_0)$, we have
\[
\abs{x-z} \ge \abs{x-y}-\abs{z-y_0}-\abs{y-y_0} \ge \tfrac14 \abs{x-y}.
\]

Applying \eqref{eq0900thu}, we obtain
\begin{equation}			\label{eq0905thu}
G(x,z) \le 4^{d-2}C_0 \abs{x-y}^{2-d},\quad \forall z \in \Omega \cap B_{R/2}(y_0).
\end{equation}
Furthermore, since $u=G^*(\cdot, x)$ satisfies
\[
L^* u=0 \;\text { in }\; \Omega\cap B_{R/2}(y_0),
\]
and $u(y_0)=G^*(y_0, x)=G(x, y_0)=0$, it follows from Remark \ref{rmk1203tue} that
\begin{equation}		\label{eq1718tue}
\sup_{0<s \le r} \,\frac{1}{s} \norm{u}_{L^\infty(B_s(y_0)\cap \Omega)} \le \frac{C}{r} \fint_{B_{2r}(y_0)\cap \Omega} u,
\end{equation}
where $r=\frac14\abs{x-y} \wedge r_0$ and $r_0>0$ depends only on $d$, $\lambda$, $\Lambda$, $\omega_{\rm coef}$, and $\Omega$.

We may assume that $\mathtt{d}(y) \le r_0$, since otherwise, we have $r_0 < \mathtt{d}(y) \le \frac14 \abs{x-y}$, and \eqref{eq0901thu} follows directly from \eqref{eq0900thu} with $C=C_0 \diam(\Omega)/r_0$.

From \eqref{eq1718tue}, \eqref{eq2005tue}, \eqref{eq0905thu}, and the assumption
\[
\mathtt{d}(y) \le r \le \tfrac14 \abs{x-y}=\tfrac{1}{4}R,
\]
it follows that
\[
\frac{1}{\mathtt{d}(y)} G(x,y) \le  \frac{C}{r} \fint_{\Omega \cap B_{2r}(y_0)} G(x,z)\,dz \le \frac{C}{r}\, \abs{x-y}^{2-d}.
\]
Recalling that $r=\frac14\abs{x-y} \wedge r_0$, we observe the right-hand side is bounded by $C \abs{x-y}^{1-d}$ when $\abs{x-y} \le 4 r_0$.
If $\abs{x-y} > 4r_0$, then the right-hand side is bounded by $C(\diam \Omega/r_0) \,\abs{x-y}^{1-d}$.
Thus, the proof of \eqref{eq0901thu} is complete in both cases.

Now, we derive the upper bound in \eqref{eq1850mon} from \eqref{eq0901thu}.
As before, it suffices to consider the case where $\mathtt{d}(x)<\frac{1}{4} \abs{x-y}=\frac{1}{4}R$.
Define $v=G(\,\cdot, y)$ and choose $x_0 \in \partial\Omega$ such that $\abs{x-x_0}=\mathtt{d}(x)$.
Since $v$ satisfies
\[
L v=0 \;\text { in }\; \Omega\cap B_{R/2}(x_0), \quad v=0 \;\text{ on }\; \partial \Omega\cap B_{R/2}(x_0),
\]
and the  gradient estimate via the standard $L^p$ estimate gives
\begin{equation}		\label{eq0902thu}
\abs{v(x)}=\abs{v(x)-v(x_0)} \le  C\mathtt{d}(x) R^{-1} \norm{v}_{L^\infty(\Omega\cap B_{R/2}(x_0))}.
\end{equation}

By the triangle inequality, for $z \in \Omega \cap B_{R/2}(x_0)$, we have
\[
\tfrac{1}{4} \abs{x-y} \le \abs{z-y} \le \tfrac{7}{4} \abs{x-y}.
\]
Using this together with \eqref{eq0902thu} and \eqref{eq0901thu}, we obtain
\[
G(x,y) \le \frac{C \mathtt{d}(x)}{\abs{x-y}^{d-1}}\left(1 \wedge \frac{\mathtt{d}(y)}{\abs{x-y}}\right).
\]
Thus, the upper bound \eqref{eq0943wed} is established.

\subsection*{The lower bound}
To proceed, we apply the Harnack inequality for nonnegative solutions of $L^*u=0$, as established in \cite[Theorem 4.3]{GK24}, along with the Harnack inequality for nonnegative solutions of $Lu=0$.
For the case $c=0$, we invoke \cite[Theorem 3.1]{S1}.
For general $c \le 0$, we apply \cite[Proposition 3.3]{DKK25} to reduce the problem to the case $c=0$; see \cite{DKK25} for details. 

Additionally, we use the following results:

\begin{enumerate}[label=(\roman*)]
\item		\label{item1}
There exists constant $r_0>0$ and $c_0 \ge 1$, depending only on the $C^{1,\alpha}$ characteristic of $\partial \Omega$, such that for any $r \in (0,r_0]$ and $x \in \Omega$ with $\mathtt{d}(x)<r$, there exists $\tilde x \in \Omega$ satisfying
\[
\abs{x-\tilde x} < c_0r\quad\mbox{and}\quad \mathtt{d}(\tilde x) \ge r.
\]

\item (Harnack chain condition)		\label{hc}
There exist constants $M>0$ and $r_0>0$ such that if $\epsilon>0$, $r \in (0,\frac{1}{4} r_0)$, $x_0 \in \partial \Omega$, and $x_1$, $x_2 \in B_r(x_0) \cap \Omega$ with $\abs{x_1-x_2} \le 2^k \epsilon$ and $\mathtt{d}(x_i) \ge \epsilon$ for $i=1,2$, there exists a chain of $Mk$  balls $B_1,\ldots, B_{Mk}$ in $\Omega$ connecting $x_1$ to $x_2$ ($x_1 \in B_1$, $x_2 \in B_{Mk}$) satisfying
\[
\diam B_j \simeq \dist(B_j, \partial\Omega),\quad  C \diam B_j  \ge  \min \{ \dist(x_1, B_j), \dist(x_2, B_j)\}
\]
for some $C>1$.

\item		\label{item2}
By \eqref{eq_bdry_hnk} and the flattening method described at the beginning of Section \ref{sec2}, we conclude that there exist constants $r_0>0$ and $C>0$, depending only on $d$, $\lambda$, $\Lambda$,  $\omega_{\mathbf A}$, $p_0$, $\norm{\vec b}_{L^{p_0}}$, $\norm{c}_{L^{p_0/2}}$, and the $C^{1,\alpha}$ characteristic of $\partial\Omega$, such that
\begin{equation}			\label{eq1100thu}
\frac{u(z)}{\mathtt{d}(z)} \ge \frac{C}{r}  \fint_{B_r(y_0)\cap \Omega} u, \quad \forall z \in B_r(y_0)\cap\Omega,\;\; r \in (0, \tfrac{1}{2}r_0],
\end{equation}
whenever and $u$ is a nonnegative solution of $L^*u=0$ in $B_{r_0}(y_0)\cap \Omega$ with $u=0$ on $B_{r_0}(y_0)\cap \partial\Omega$ for some $y_0 \in \partial\Omega$.

\item		\label{item3}
Similar to \ref{item2}, there exist constants $r_0>0$ and $C>0$ that depend only on $d$, $\lambda$, $\Lambda$,  $p_0$, $\norm{\vec b}_{L^{p_0}}$, $\norm{c}_{L^{p_0/2}}$, and the $C^{1,\alpha}$ characteristic of $\partial\Omega$, such that \eqref{eq1100thu} holds for all nonnegative solutions of $Lu=0$ in $B_{r_0}(y_0)\cap \Omega$ with $u=0$ on $B_{r_0}(y_0)\cap \partial\Omega$ for some $y_0 \in \partial\Omega$.
See, for instance, \cite{N1}.

\item
The constants $r_0$ in \ref{item1} -- \ref{item3} are not necessarily the same.
However, we take the smallest of the three and denote it by $r_0$.
\end{enumerate}

\medskip

Our goal is to establish the following lower bound:
\begin{equation}		\label{eq0943thu}
G(x,y) \ge \frac{C}{\abs{x-y}^{d-2}} \left(1 \wedge \frac{\mathtt{d}(y)}{\abs{x-y}}\right) \left(1 \wedge \frac{\mathtt{d}(x)}{\abs{x-y}}\right),\quad x\neq y.
\end{equation}

\medskip
\noindent
\emph{Case 1: }$\abs{x-y} \le \frac{1}{3} \mathtt{d}(y)$.
\smallskip

Setting $r=\frac14 \abs{x-y}$, we observe that $B_{5r}(x) \subset \Omega$.
Consider a nonnegative function $\eta \in C^\infty_c(B_{4r}(x))$ such that $\eta=1$ in $B_{3r}(x)$, with $\norm{D\eta}_{L^\infty} \le 2/r$, $\norm{D^2 \eta}_{L^\infty} \le 4/r^2$.
Using $c\le 0$, H\"older's inequality, and the assumption that $p_0>d$, we obtain
\begin{align*}
1&=\eta(x)=\int_{\Omega} G(x, \cdot\,) L \eta \le \int_{B_{4r}(x)\setminus B_{3r}(x)} G^*(\,\cdot,x) \left(a^{ij}D_{ij}\eta + b^i D_i \eta \right)\\
&\le \sup_{B_{4r}\setminus B_{3r}} G^*(\,\cdot,x) \left( \frac{C}{r^2}\, \abs{B_{4r}} +\frac{C}{r} \,\norm{\vec b}_{L^{p_0}} \abs{B_{4r}}^{1-1/p_0}\right)\\
& \le C r^{d-2} \sup_{B_{4r}\setminus B_{3r}} G^*(\,\cdot, x),
\end{align*}
where $C$ depends on $d$, $p_0$, and $\diam \Omega$.

Since any two points in $B_{4r}(x)\setminus B_{3r}(x)$ can be connected by a chain of at most $N(d)$ balls of radius $r$, all contained in $B_{5r}(x)\setminus B_{2r}(x) \subset \Omega$, applying the Harnack inequality iteratively for $G^*(\,\cdot, x)$ gives
\[
G(x,y)= G^*(y,x) \ge C \sup_{B_{4r}\setminus B_{3r}} G^*(\,\cdot, x).
\]

Thus, we have
\begin{equation}			\label{eq2100thu}
G(x,y) \ge C \abs{x-y}^{2-d},
\end{equation}
where $C$ depends on $d$, $\lambda$, $\Lambda$, $p_0$, $\norm{\vec b}_{L^{p_0}}$, $\norm{c}_{L^{p_0/2}}$, and $\omega_{\mathbf A}$.
Since
\[
\mathtt{d}(x) \ge \mathtt{d}(y)-\abs{x-y}\ge 2\abs{x-y},
\]
we establish \eqref{eq0943thu}.

\medskip
\noindent
\emph{Case 2: }$\frac{1}{3} \mathtt{d}(y)<\abs{x-y} < 7 c_0 \mathtt{d}(y)$.
\smallskip

Take $r= \abs{x-y} \wedge r_0$ and choose $x_1 \in \Omega$ such that $\abs{y-x_1}=(1/21c_0)r$.
Note that 
\[
\mathtt{d}(y)> \frac{1}{7c_0}r, \qquad \mathtt{d}(x_1) \ge \mathtt{d}(y)-\frac{1}{21c_0} r \ge \frac{2}{21c_0} r.
\]
Applying \eqref{eq2100thu} with $x_1$ in place of $x$, we obtain
\begin{equation}			\label{eq2200thu}
G(x_1,y) \ge C r^{2-d}.
\end{equation}

If $\mathtt{d}(x) < (1/3c_0)r$, using \ref{item1}, we choose $\tilde x \in\Omega$ such that
\[
\mathtt{d}(\tilde x) \ge \frac{1}{3c_0} r\quad \text{and}\quad \abs{x-\tilde x}<\frac{1}{3}r.
\]
If $\mathtt{d}(x) \ge (1/3c_0)r$, we simply take $\tilde x=x$.
In either case, observe that
\[
\abs{x_1-\tilde x} \le \abs{x-y}+\abs{y-x_1}+\abs{x-\tilde x} < \abs{x-y}+r.
\]
Thus, when $r=\abs{x-y}$, we obtain $\abs{x_1-\tilde x} < 2r$.
On the other hand, if $r=r_0$, we have
\[
\abs{x_1-\tilde x} \le (\diam\Omega/r_0)\, r.
\]

Using \ref{hc}, we apply the Harnack inequality iteratively to $G(\,\cdot, y)$ to compare $G(x_1,y)$ with $G(\tilde x,y)$.
Using \eqref{eq2200thu}, we obtain
\begin{equation}			\label{eq1749sun}
G(\tilde x, y) \ge C G(x_1,y) \ge C r^{2-d},
\end{equation}
where the constant $C>0$ depends only on $d$, $\lambda$, $\Lambda$, $p_0$, $\norm{\vec b}_{L^{p_0}}$, $\norm{c}_{L^{p_0/2}}$, $\omega_{\mathbf A}$,  $\diam \Omega$, and the $C^{1,\alpha}$ characteristic of $\partial \Omega$.

Next, we derive an estimate to compare $G(\tilde x, y)$ with $G(x, y)$.
If $\mathtt{d}(x) \ge (1/3c_0)r$, then we have $\tilde x=x$.
Recalling that $r=\abs{x-y}\wedge r_0$, it follows from \eqref{eq1749sun} that
\[
G(x,y) =G(\tilde x,y) \ge C \abs{x-y}^{2-d}.
\]
This immediately yields \eqref{eq0943thu}.

If $\mathtt{d}(x)<(1/3c_0)r$, we choose $x_0 \in \partial\Omega$ such that $\abs{x-x_0}=\mathtt{d}(x)$.
Since
\[
2 B_{r/6c_0}(\tilde x)=B_{r/3c_0}(\tilde x) \subset B_r(x_0) \cap \Omega,
\]
applying \ref{item3} to $G(\,\cdot, y)$ and using the Harnack inequality for $G(\,\cdot, y)$, we obtain
\begin{equation}			\label{eq1242sun}
G(x,y) \ge \frac{C \mathtt{d}(x)}{r} \fint_{B_{r/6c_0}(\tilde x)} G(z,y)\,dz \ge \frac{C \mathtt{d}(x)}{r}  G(\tilde x, y).
\end{equation}
Combining \eqref{eq1242sun} with \eqref{eq1749sun} and recalling that $r=\abs{x-y}\wedge r_0$, we obtain \eqref{eq0943thu}.

\medskip
\noindent
\emph{Case 3: }$\mathtt{d}(y) \le \abs{x-y}/7c_0$.
\smallskip
 
We first consider the case when $\mathtt{d}(y) \ge (1/7c_0) r_0$.
Choose points  $x_1$ such that
\[
\abs{y-x_1}=(1/21c_0) r_0.
\]
If $\mathtt{d}(x) < (1/3c_0)r_0$, then using \ref{item1}, choose $\tilde x \in \Omega$ such that
\[
\mathtt{d}(\tilde x) \ge (1/3c_0)r_0 ,\quad \abs{x-\tilde x}< r_0/3.
\]
If $\mathtt{d}(x) \ge (1/3c_0)r_0$, we take $\tilde x=x$.
Note that the following estimates hold:
\[
\mathtt{d}(x_1) \ge (2/21c_0) r_0, \quad G(x_1,y) \ge C r_0^{2-d}, \quad \abs{x_1-\tilde x} \le (\diam \Omega/r_0)\,r_0,
\]
Observing that $\mathtt{d}(x_1)\ge (2/21c_0) r_0$ and utilizing \ref{hc}, we obtain, similar to \eqref{eq1749sun},
\begin{equation}			\label{eq2117tue}
G(\tilde x, y) \ge C G(x_1, y) \ge  C r_0^{2-d}.
\end{equation}
As before, we need to compare $G(\tilde x, y)$ with $G(x, y)$.
If $\mathtt{d}(x) \ge r_0/3c_0$, then $\tilde x=x$.
Using the fact that $r_0 \le 7c_0 \mathtt{d}(y) \le \abs{x-y}$, we obtain from \eqref{eq2117tue}
\[
G(x,y) =G(\tilde x,y) \ge C \abs{x-y}^{2-d}.
\]
This immediately yields \eqref{eq0943thu}.

If $\mathtt{d}(x)<(1/3c_0)r_0$, we choose $x_0 \in \partial\Omega$ such that $\abs{x-x_0}=\mathtt{d}(x)$.
Since
\[
2B_{r_0/6c_0}(\tilde x)=B_{r_0/3c_0}(\tilde x) \subset B_{r_0}(x_0) \cap \Omega,
\]
applying \ref{item3} to $G(\,\cdot, y)$ and using the Harnack inequality for $G(\,\cdot, y)$, we obtain
\begin{equation}			\label{eq2111tue}
G(x,y) \ge \frac{C \mathtt{d}(x)}{r_0} \fint_{B_{r_0/6c_0}(\tilde x)} G(z,y)\,dz \ge \frac{C \mathtt{d}(x)}{r_0}  G(\tilde x, y).
\end{equation}
Combining \eqref{eq2117tue} and \eqref{eq2111tue}, and using $r_0 \le 7c_0 \mathtt{d}(y) \le \abs{x-y}$, we obtain
\[
G(x,y) \ge C \mathtt{d}(x) r_0^{1-d} \ge \frac{C \mathtt{d}(x)}{\abs{x-y}^{d-1}}.
\]
Obviously, this implies \eqref{eq0943thu}.

\smallskip
Now, we assume that $\mathtt{d}(y) < (1/7c_0)r_0$.
Let $y_0 \in \partial \Omega$ be such that $\abs{y-y_0}=\mathtt{d}(y)$.
Define $r=\abs{x-y} \wedge r_0$.
Since $\mathtt{d}(y) < (1/7c_0)r \le r/7$, we apply \ref{item2} to $u=G^*(\,\cdot, x)$, obtaining
\begin{equation}			\label{eq1105thu}
\frac{G^*(y,x)}{\mathtt{d}(y)} \ge \frac{C}{r} \fint_{B_{3r/7}(y_0)\cap \Omega} G^*(z,x)\,dz.
\end{equation}

Using \ref{item1}, we choose $\tilde y \in\Omega$ such that
\[
\mathtt{d}(\tilde y) \ge (1/7c_0) r\quad \text{and}\quad \abs{y-\tilde y}<r/7.
\]
Since
\[
2B_{r/14c_0}(\tilde y) =B_{r/7c_0}(\tilde y) \subset B_{3r/7}(y_0)\cap \Omega,
\]
from \eqref{eq1105thu}, the Harnack inequality for $G^*(\,\cdot,y)$, and \eqref{eq2005tue}, we obtain
\begin{equation}			\label{eq1634fri}
G(x,y) \ge \frac{C \mathtt{d}(y)}{r} \fint_{B_{r/14c_0}(\tilde y)} G^*(z,x)\,dz \ge \frac{C \mathtt{d}(y)}{r}  G(x, \tilde y).
\end{equation}
Choose a point $x_1 \in \Omega$ such that $\abs{x_1-\tilde y}=(1/21c_0)r$.
From \eqref{eq2100thu}, it follows that
\begin{equation}			\label{eq2148fri}
G(x_1,\tilde y) \ge C r^{2-d}.
\end{equation}

If $\mathtt{d}(x) < (1/7c_0)r$, using \ref{item1}, we choose $\tilde x \in\Omega$ such that
\[
\mathtt{d}(\tilde x) \ge (1/7c_0)r\quad \text{and} \quad \abs{x-\tilde x}<r/7.
\]
If $\mathtt{d}(x) \ge (1/7c_0)r$, we simply take $\tilde x=x$.
Observe that
\[
\abs{x_1-\tilde x} \le \abs{x-y}+\abs{y-\tilde y}+\abs{\tilde y-x_1}+\abs{x-\tilde x} < \abs{x-y}+r.
\]
Thus, when $r=\abs{x-y}$, we obtain $\abs{x_1-\tilde x} \le 2r$.
On the other hand, if $r=r_0$, we have
\[
\abs{x_1-\tilde x} \le (\diam\Omega/r_0)\, r.
\]

Since $\mathtt{d}(x_1)\ge (2/21c_0)r$, applying the Harnack inequality to $G(\,\cdot, \tilde y)$ and utilizing \ref{hc}, we obtain from \eqref{eq2148fri}, as before,
\[
G(\tilde x, \tilde y) \ge C G(x_1, \tilde y) \ge C r^{2-d}.
\]
Similar to \eqref{eq1242sun}, we compare $G(\tilde x, \tilde y)$ with $G(x,\tilde y)$ as follows:
\[
G(x,\tilde y) \ge \frac{C \mathtt{d}(x)}{r} \fint_{B_{r/9c_0}(\tilde x)} G(z, \tilde y)\,dz \ge \frac{C \mathtt{d}(x)}{r}  G(\tilde x, \tilde y).
\]

Combining the previous two estimates with \eqref{eq1634fri}, we obtain \eqref{eq0943thu}.

\medskip

Having verified all cases, the proof of \eqref{eq0943thu} is complete.
\qed

\section{Counterexample}				\label{counter_ex}
We aim to construct a counterexample in $\bR^2$ where the coefficient matrix $\mathbf A=(a_{ij})$ does not belong to the DMO class.
Specifically, we seek a radial function $\eta=\eta(r)$ satisfying
\[
\eta(r)=\frac{1}{\log r}(1+o(1))\quad \text{as }\; r\to 0+,
\]
such that the function
\[
u(x)=u(x_1,x_2):=-\frac{x_2}{\log \,\abs{x}}
\]
is solution to $L_0^* u=D_{ij}(a^{ij}u)=0$ in $B_{\varepsilon}^+=B_{\varepsilon}^+(0)$ for some $\varepsilon \in (0,\frac{1}{2})$, where
\[
a_{ij}(x)=\delta_{ij}+ \frac{x_ix_j}{\abs{x}^2}\, \eta(\abs{x}).
\]

It is straightforward to verify that $a_{12}(x)$ is not a DMO function.
Additionally, the coefficients matrix $\mathbf A$ satisfies condition \eqref{ellipticity-nd} in $B_{\varepsilon}^+$ provided that
\begin{equation}	\label{eq1141mon}
\sup_{0<r< \varepsilon} \,\abs{\eta(r)} <1.
\end{equation}
Moreover, it is easy to see that $u \in C(\overline{B_\varepsilon^+})$ and that $u=0$ on $T_\varepsilon=B_\varepsilon^+ \cap \{x_2=0\}$.
However, since
\[
\frac{\partial u}{\partial x_2}=-\frac{1}{\log \abs{x}}+\frac{x_2^2}{(\log \abs{x})^2 \abs{x}^2},
\]
we find that $\partial u/\partial x_2(0)=0$, which contradicts the Hopf lemma at $0$.

A direct computation shows that $u$ is a solution of $L_0^*u=0$ in $B_\varepsilon^+$ if and only if $\eta(r)$ satisfies the following ODE in $(0, \varepsilon)$:
\begin{equation}			\label{eq1238thu}
r^2 \eta''(r)+2r\left(2-\frac{1}{\log r}\right) \eta'(r)+\left(\frac{2}{(\log r)^2}-\frac{3}{\log r}+2\right)\eta(r)-\frac{2}{\log r}+\frac{2}{(\log r)^2}=0.
\end{equation}

Using the change of variables $t=-\log r$ and setting
\[
y(t)=y(-\log r)=\eta(r),
\]
we obtain the transformed equation:
\begin{equation}			\label{eq1135thu}
y''(t)-\left(3+\frac{2}{t}\right)y'(t)+\left(2+\frac{3}{t}+\frac{1}{t^2}\right)y(t)+\frac{2}{t}+\frac{2}{t^2}=0,\quad   t\in (-\log \varepsilon, \infty).
\end{equation}
Rearranging, we rewrite the equation as
\[
y''(t)-3y'(t)+2y(t)=\frac{2}{t}\,y'(t)-\left(\frac{3}{t}+\frac{1}{t^2}\right)y(t)-\frac{2}{t}-\frac{2}{t^2}.
\]

To find a solution $y(t)$, we introduce a function $f(t)$ belonging to a suitable function class $X$ and apply a fixed point method to solve the following modified equation:
\begin{equation}		\label{eq0853thu}
y''(t)-3y'(t)+2y(t)=\frac{2}{t}\,f'(t)-\left(\frac{3}{t}+\frac{1}{t^2}\right)f(t)-\frac{2}{t}-\frac{2}{t^2}.
\end{equation}

Using the standard method for solving second-order linear ODEs, we obtain the following representation for all solutions:
\begin{multline*}
y(t)=-e^t \int^t e^{-s} \left\{\frac{2}{s}\,f'(s)-\left(\frac{3}{s}+\frac{1}{s^2}\right)f(s)-\frac{2}{s}-\frac{2}{s^2} \right\}\,ds\\
+e^{2t} \int^t e^{-2s}\left\{\frac{2}{s}\,f'(s)-\left(\frac{3}{s}+\frac{1}{s^2}\right)f(s)-\frac{2}{s}-\frac{2}{s^2} \right\}\,ds. 
\end{multline*}

Integrating by parts, we obtain
\[
y(t)=e^t \int^t e^{-s} \left\{\left(\frac{1}{s}-\frac{1}{s^2}\right)f(s)+\frac{2}{s}+\frac{2}{s^2}\right\} \,ds
+e^{2t} \int^t  e^{-2s}\left\{\left(\frac{1}{s}+\frac{1}{s^2}\right)\,f(s)-\frac{2}{s}-\frac{2}{s^2} \right\}\,ds. 
\]

This leads us to consider the mapping $f \mapsto Tf$, where
\[
Tf(t)=- \int_t^{\infty} e^{t-s} \left\{\left(\frac{1}{s}-\frac{1}{s^2}\right)f(s)+\frac{2}{s}+\frac{2}{s^2}\right\} + e^{2t-2s}\left\{\left(\frac{1}{s}+\frac{1}{s^2}\right)\,f(s)-\frac{2}{s}-\frac{2}{s^2} \right\}\,ds.
\]
It is clear that $Tf(t)$ is a solution of \eqref{eq0853thu}.
Define the Banach space
\[
X=\left\{f \in C(I): \norm{f}_X=\sup_{t \in I} t\abs{f(t)}<\infty\right\}, \quad I:= (-\log \varepsilon, \infty).
\]
We will show that $T$ is a contraction mapping in $X$.
Indeed, we have
\begin{align*}
t \Abs{Tf(t)-Tg(t)} &\le t \int_t^{\infty} \Abs{\,e^{t-s} \left(\frac{1}{s}-\frac{1}{s^2}\right)+ e^{2t-2s} \left(\frac{1}{s}+\frac{1}{s^2}\right)\,} \frac{1}{s} \norm{f-g}_X \,ds\\
&\le t \norm{f-g}_X \left\{ \int_t^{4t/3}\frac{2}{s} \left(\frac{1}{s}+\frac{1}{s^2}\right) \,ds +e^{-t/3} \int_{4t/3}^\infty \frac{2}{s} \left(\frac{1}{s}+\frac{1}{s^2}\right) \,ds\right\}\\
&\le\norm{f-g}_X \left\{ \left(\frac{1}{2}+\frac{7}{16t}\right) +e^{-t/3}\left(\frac{3}{2}+\frac{9}{16t}\right)\right\}.
\end{align*}

Thus, $T$ is a contraction mapping, provided that $\varepsilon \in (0,\frac{1}{2})$ is sufficiently small.
By starting with $f(t)=1/t$, we observe that $T^n f$ converges to some function $y$ in $X$.
Furthermore, by L'Hospital's rule,
\[
y(t)=\frac{1}{t} (-1+o(1))\quad\text{as }\;t \to \infty.
\]
Indeed,
\begin{align*}
\lim_{t\to \infty}\frac {-\int_t^{\infty} e^{t-s} \left\{\left(\frac{1}{s}-\frac{1}{s^2}\right)y(s)+\frac{2}{s}+\frac{2}{s^2}\right\}\,ds}{\frac{1}{t}}
&=\lim_{t\to \infty}\frac {-\int_t^{\infty} e^{-s} \left\{\left(\frac{1}{s}-\frac{1}{s^2}\right)y(s)+\frac{2}{s}+\frac{2}{s^2}\right\}\,ds}{e^{-t}\frac{1}{t}}\\
&=\lim_{t\to \infty}\frac {e^{-t} \left\{\left(\frac{1}{t}-\frac{1}{t^2}\right)y(t)+\frac{2}{t}+\frac{2}{t^2}\right\}}{-e^{-t}\left(\frac{1}{t}+\frac{1}{t^2}\right)}=-2
\end{align*}
and similarly
\[
\lim_{t\to +\infty}-t\int_t^{\infty}  e^{2t-2s}\left\{\left(\frac{1}{s}+\frac{1}{s^2}\right)\,f(s)-\frac{2}{s}-\frac{2}{s^2} \right\}\,ds=1.
\]
Clearly, $y$ satisfies \eqref{eq1135thu}, which implies that $\eta(r)=y(-\log r)$ is a solution to \eqref{eq1238thu}.
Moreover, since
\[
\sup_{-\log \epsilon<t< \infty} \abs{y(t)}= \sup_{0<r<\epsilon} \abs{\eta(r)},
\]
it follows that \eqref{eq1141mon} is satisfied by choosing $\varepsilon$ sufficiently small.
\qed


\end{document}